\documentclass[a4paper,11pt,english]{amsart}
\usepackage{amsmath,amssymb,amsfonts}
\usepackage{epsfig,color}
\usepackage{mathrsfs}
\usepackage{dsfont}
\usepackage{dsfont,stmaryrd}

\usepackage[utf8]{inputenc}
\usepackage{geometry}
\usepackage[english]{babel}
\usepackage{enumitem}
\usepackage{hyperref}
\usepackage{mathtools}

\newtheorem{thm}{Theorem}[section]
\newtheorem{prop}{Proposition}[section]
\newtheorem{cor}{Corollary}[section]
\newtheorem{rmk}{Remark}[section]
\newtheorem{lma}{Lemma}[section]
\newtheorem*{cond}{Condition $(\star)$}

\newcommand{\E}{\mathbb{E}}

\def\N{{\rm I\kern-0.16em N}}
\def\R{{\rm I\kern-0.16em R}}
\def\E{{\rm I\kern-0.16em E}}
\def\P{{\rm I\kern-0.16em P}}
\def\F{{\rm I\kern-0.16em F}}
\def\B{{\rm I\kern-0.16em B}}
\def\C{{\rm I\kern-0.46em C}}
\def\G{{\rm I\kern-0.50em G}}

\title[On the zeros of non-analytic random periodic signals]{On the zeros of non-analytic \\ random periodic signals}
\author{J\"urgen Angst}
\email{jurgen.angst@univ-rennes1.fr}
\urladdr{http://www.angst.fr}

\thanks{This work was supported by the ANR grant UNIRANDOM, ANR-17-CE40-0008.}

\author{Guillaume Poly,} 
\email{guillaume.poly@univ-rennes1.fr} 
\urladdr{https://sites.google.com/site/guillaumejpoly/home}

 \address{Univ Rennes, CNRS, IRMAR - UMR 6625, F-35000 Rennes, France}

\begin{document}
\textcolor{white}{bla}\par
\vspace{-0.8cm}
\begin{abstract}
In this paper, we investigate the local universality of the number of zeros of a random periodic signal of the form $S_n(t)=\sum_{k=1}^n a_k f(k t)$, where $f$ is a $2\pi-$periodic function satisfying weak regularity conditions and where the coefficients $a_k$ are i.i.d. random variables, that are centered with unit variance. 
In particular, our results hold for continuous piecewise linear functions. 
We prove that the number of zeros of $S_n(t)$ in a shrinking interval of size $1/n$ converges in law as $n$ goes to infinity to the number of zeros of a Gaussian process whose explicit covariance only depends on the function $f$ and not on the common law of the random coefficients $(a_k)$. 
As a byproduct, this entails that the point measure of the zeros of $S_n(t)$ converges in law to an explicit limit on the space of locally finite point measures on $\mathbb R$ endowed with the vague topology. 
The standard tools involving the regularity or even the analyticity of $f$ to establish such kind of universality results are here replaced by some high-dimensional Berry-Esseen bounds recently obtained in \cite{chernozhukov2017central}. The latter allow us to prove functional CLT's in $C^1$ topology in situations where usual criteria can not be applied due to the lack of regularity. 
\end{abstract}


\maketitle
\setcounter{tocdepth}{2}

\par
\vspace{-0.5cm}
\tableofcontents

\newpage

\section[Introduction]{Introduction and statement of the results}
The study of roots and level sets of random functions is a well established topic in probability theory with numerous connections with other domains of mathematics and mathematical physics. In univariate settings,  it is mostly focused on the wide family of random algebraic or trigonometric polynomials and the related literature is truly considerable. A recurrent thematic in this domain of research is the so-called \textit{universality}, which we formalize a bit below.

\subsection{Universality properties of random nodal sets}
Let us consider $\{a_k\}_{k\ge 1}$ a sequence of i.i.d. standard Gaussian random variables, a deterministic sequence of real numbers $\{c_k\}_{k\ge 1}$, a family of real functions $(f_k)_{k \geq 1}$ and define the random function 
\[
S_n(t):=\sum_{k=1}^n a_k c_k f_k(t).
\]
When $f_k(t)=t^k$ or $f_k(t)=\cos(kt)$, we recover the standard models of Gaussian random algebraic or trigonometric polynomials. The random variable of interest is here the number of zeros of $S_n$ in a given interval, i.e.
$$\mathcal{N}_n(a,b):=\text{Card}  \left\{t\in [a,b]\,\Big{|}\,S_n(t)=0 \right\}.$$

Now, given $\{\widetilde{a}_k\}_{k\ge 1}$ another sequence of i.i.d. standard random variables not necessarily Gaussian we set in the same way
\[
\widetilde{S}_n(t):=\sum_{k=1}^n  \widetilde{a}_k c_k f_k(t), \qquad  \widetilde{\mathcal{N}}_n(a,b):=\text{Card}  \left\{t\in [a,b]\,\Big{|}\,\widetilde{S}_n(t) \right\}.
\]

Of course, since $a_1$ is not equal in distribution to $\widetilde{a}_1$, the random variables $\mathcal{N}_n(a,b)$ and $\widetilde{\mathcal{N}}_n(a,b)$ have different distributions. Nevertheless, when the degree $n$ tends to infinity, that is to say when the amount of noise in the system increases, both random variables $\mathcal{N}_n(a,b)$ and $\widetilde{\mathcal{N}}_n(a,b)$, their expectations or variances, may display a similar asymptotic behavior. One then says that this behavior is universal since it does not depend on the distribution of the chosen random coefficients.
\par
\medskip
After a proper renormalization, this universal type of behavior can occur at a global scale, i.e. on a macroscopic interval $[a,b]$, but also at a microscopic scale, namely on shrinking intervals whose size goes to zero as $n$ goes to infinity. This type of microscopic behavior is then customarily called a local universality phenomenon. 
We refer to \cite{MR3439098} and the references therein for universality results in the case of real roots of various families of algebraic polynomials.
Similar questions for random trigonometric polynomials have been explored in \cite{flasche,iksanov2016,nousAMS,chang2018}.
We finally refer to \cite{nguyen2017roots, MR3846831,flasche2,kabzap2014} for more general conditions entailing analogue universality phenomena for roots of random analytic functions.
\par
\medskip
In all the aforementioned references, the analyticity of the random functions under consideration plays a central role, either via zeros counting formula such as Jensen formula, or via anti-concentration arguments that require an a priori unbounded number of derivatives.
\par
\medskip
In this article, we are precisely interested in establishing local universality results in a non-analytic context.  For this, we introduce and study a simple model of random periodic and possibly non-analytic functions which naturally generalizes random trigonometric polynomials. Namely, we consider the case where $f_k(t)=f(kt)$ and $f$ is a 
continuous $2\pi-$periodic function, which is piecewise $C^1$. 
In particular, this covers the case of random superpositions of triangular signals, which can be seen as a natural non-regular alternative to random trigonometric polynomials, see Figure \ref{fig.1} below. 

\begin{figure}[ht] 
\begin{center}\includegraphics[scale=0.5]{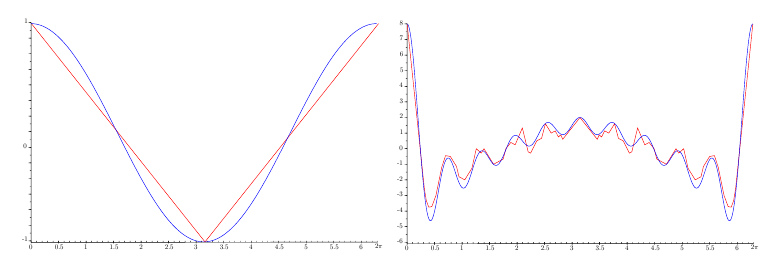}\end{center}
\caption{Cosine and triangular input functions $f$ (left) and a simulation of the corresponding random superpositions $S_n$ (right), for $n=10$ and Rademacher entries. On this particular example, the number of roots of the red triangular signal is twice the one of the trignometric blue one.} \label{fig.1}
\end{figure}

We would like to emphasize here that one can hardly go beyond the above regularity assumptions on $f$ since there are H\"older functions with unbounded number of roots on any interval. 
To the best of our knowledge, the question of universality for non-smooth models has not been investigated yet.

\subsection{Presentation of the model and main results} \label{sec.state}

We specify here our base model of random periodic signals and state the main results of the paper. Let $f$ be a continuous $2\pi-$periodic function that is almost everywhere differentiable with derivative $f'$. To avoid pathological behaviors, we require that  both
\[
\langle f, f \rangle >0 \quad \text{ and } \quad \langle f', f' \rangle >0, 
\]
where $\langle ., . \rangle$ denotes the usual $\mathbb L^2$ scalar product in $\mathbb L^2([0, 2\pi])$.
Now let $(a_k)_{k \geq 1}$ be a sequence of real random variables that are independent with common distribution such that $\mathbb E[a_1]=0$ and $\mathbb E[a_1^2]=0$. The main object of this article is then the study of the set of zeros of the random periodic function  $\sum_{k=1}^n a_k   f( k t)$ and its asymptotics as $n$ goes to infinity. 
Since we are primarily interested in the local universality properties of this random nodal set, i.e. the properties of the nodal set at microscopic scale, we normalize and localize the process at scale $1/n$, i.e. we rather consider the rescaled process 
\begin{equation}\label{eq.defX}
X_n(t)  :=\displaystyle{\frac{1}{\sqrt{n}} \sum_{k=1}^n a_k   f\left( \frac{k(p_n+t)}{n} \right)}, \quad t \in \mathbb{R},
\end{equation}
where $p_n$ be a sequence of integers such that $p_n/n \to \alpha \in (0,2\pi) \backslash \pi \mathbb Q$ as $n$ goes to infinity. In other words, we look at the zeros of the original random periodic signal in a window of size $2\pi/n$ around a point $\alpha\in  (0,2\pi) \backslash \pi \mathbb Q$, see Figure 1 below.

\begin{figure}[ht]
\begin{center}\includegraphics[scale=0.5]{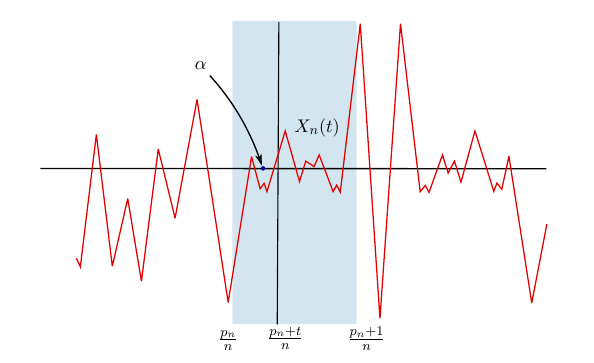}\end{center}
\caption{Localization of the signal at a microscopic scale.}
 \end{figure}
 The process $(X_n(t))_{t \in \mathbb{R}}$ is then naturally almost everywhere differentiable with
\begin{equation}\label{eq.defXprime}
X_n'(t)  =\displaystyle{\frac{1}{\sqrt{n}} \sum_{k=1}^n a_k   \times \frac{k}{n} \times  f'\left( \frac{k(p_n+t)}{n} \right)}.
\end{equation}

For some of the statements below, we will need to impose moreover that the limit point $\alpha$ is Diophantine and the sequence $p_n/n$ converges to $\alpha$ at polynomial speed. To be clear when this extra condition is needed or not, let us  introduce it formally as the following condition.

\begin{cond}We will say that $(\alpha,p_n)$ satisfies condition $(\star)$ if $\alpha/\pi$ is a Diophantine number, i.e. there exists $c_\alpha>0$ and $\nu_\alpha \in ]0,1[$ such that for every $(p,q)\in\mathbb{N} \times \mathbb N^*$ 
\begin{equation}\label{Def-Dioph}
\left|\frac{\alpha}{\pi}-\frac{p}{q}\right|\ge \frac{c_\alpha}{q^{2+\nu_\alpha}},
\end{equation}
and if $p_n/n$ converges to $\alpha$ at some polynomial speed, namely
\begin{equation}\label{conv-dioph}
\exists c_1,c_2>0,\,\,\text{s.t.}\,\,\left|\frac{p_n}{n}-\alpha\right|\le \frac{c_1 }{n^{c_2}}.
\end{equation}
\end{cond}

\begin{rmk}
Let us emphasize that almost all irrational or Diophantine, and that the condition on the polynomial speed of convergence is naturally satisfied in the simple case where $p_n := \lfloor n \alpha \rfloor$ for all $n \geq 1$.
\end{rmk}

Let us now present the main results of this paper. First, they consist in showing that, under various mild regularity assumptions on the input function $f$ and for every compact $[a,b]\subset \mathbb{R}$, the process $(X_n(t))_{t \in [a, b]}$ converges in distribution in the $C^1([a, b])$ topology to an explicit stationary Gaussian process $(X_{\infty}(t))_{t \in [a, b]}$. Setting $\check{f}(x):=f(-x)$ for all $x \in \mathbb R$, we indeed have the following convergence.
 
\begin{thm}\label{theo.C2}
Let $(a_k)_{k \geq 1}$ be a sequence of independent and identically distributed random variables, which are centered  with unit variance. Let $f$ be a continuous and $2\pi-$periodic function, which is piecewise $C^1$. Let $p_n$ be a sequence of real numbers such that $p_n/n$ converges to $\alpha \in (0,2\pi) \backslash \pi\mathbb Q$. Then the stochastic process $(X_n(t),X_n'(t))_{t \in \mathbb R}$ defined by Equation \eqref{eq.defX}
converges in the sense of finite marginals to a stationary and centered Gaussian process $(X_{\infty}(t), X_{\infty}'(t))_{t \in \mathbb R}$ such that, for all $s,t$  
\[
\mathbb E\left[ X_{\infty}(s) X_{\infty}(t) \right]   =: \rho(t-s), \quad \text{with} \quad \left \lbrace \begin{array}{l} \displaystyle{\rho(u):= \frac{1}{u} \int_{0}^u (f \ast \check{f})(x) dx}, \;\; u\neq 0,\\
\\
\rho(0)=\langle f, f\rangle.
\end{array}
\right.
\]
Moreover, under each of the following additional conditions
\begin{enumerate}[label=\roman*)] \setlength\itemsep{0.5em}
\item the function $f$ is $C^2$,
\item there exists $\alpha>0$ and $\beta>2$ with $\alpha \beta >1$  such that $f$ is $C^{1+\alpha}$ and $\mathbb E[|a_1|^{\beta}]<+\infty$,
\item there exists $\alpha>0$ such that $f$ is $C^{1+\alpha}$ and $\mathbb E[|a_1|^4]<+\infty$,
\item condition $(\star)$ holds, $f$ is piecewise linear and $\mathbb E[|a_1|^4]<+\infty$,
\end{enumerate}
\smallskip
the laws of $(X_n(t))_{t \in [a,b]}$ are tight with respect to the $C^1$ topology on $[a,b]$ for every compact interval $[a,b]$.

\end{thm}

Now the main byproduct of establishing the above convergence in the $C^1$ topology of every compact $[a,b]$ is that one deduces that the nodal set associated to $(X_n(t))_{t \in [a,b]}$ converges to the one of $(X_{\infty}(t))_{t \in [a,b]}$. For example, one then immediately gets the convergence in distribution of the number of zeros in a given compact set. We refer the reader to section \ref{Cvzeros} for more details.

\begin{thm}\label{theo.ConvZero}
Under the hypotheses of Theorems \ref{theo.C2} and if one of the additional conditions $i), ii), iii)$ or $iv)$ holds, for any interval $[a,b] \subset \mathbb{R}$, the number of zeros $\mathcal N(X_n,[a,b])$ of $(X_n(t))_{t \in \mathbb{R}}$ in $[a,b]$ converges in distribution to its analogue $\mathcal N(X_{\infty},[a,b])$ as $n$ goes to infinity.
\end{thm}

As in \cite{iksanov2016}, one then also gets the convergence in law of the point measure associated to the zeros of $(X_n(t))_{t \in \mathbb{R}}$, the convergence taking place in the space of locally finite point measures on $\mathbb{R}$ endowed with the vague topology. 

\begin{thm}\label{theo.ConvPointMeas}
Under the hypotheses of Theorems \ref{theo.C2} and if one of the additional conditions $i), ii), iii)$ or $iv)$ holds,, the point measure associated to the zeros of $(X_n(t))_{t \in \mathbb{R}}$ converges to the point measure associated to zeros of $(X_{\infty}(t))_{t \in \mathbb{R}}$.
\end{thm}
As it is clear on the covariance function of the limit process $(X_{\infty}(t))_{t \in \mathbb{R}}$, its distribution naturally do depend on the input function $f$. This dependence is rather explicit and allows computations for the observables associated to the limit nodal set. For example, using the celebrated Kac--Rice formula , one easily computes the expected number of zeros in any interval.

\begin{prop}\label{moy.limit}
For any $[a,b] \subset \mathbb R$, we have
\[
\mathbb E[\mathcal N(X_{\infty},[a,b])] = \frac{b-a}{\pi}  \sqrt{ \frac{1}{3} \frac{\langle f',f' \rangle}{\langle f,f \rangle}}.
\]
\end{prop}

\begin{rmk}
Going back to trigonometric polynomials, e.g. if the input function is cosine $f = \cos$, we have $\langle f,f \rangle=\langle f',f' \rangle$ and we thus recover the fact that the expected number of zeros of the  universal limit $\sin_c$ process in $[a,b] \subset \mathbb R$ is $(b-a)/\pi \sqrt{3}$.
\end{rmk}

The plan of the article in the following. The proofs of Theorem\ref{theo.C2}, i.e. the convergences in the $C^1([a, b])$ topology can be split into two parts: the convergence of finite marginals and some tightness criteria. The convergence of finite marginals is detailed in the next Section \ref{sec.marginals}, whereas the different tightness criteria depending on the regularity of $f$ are established in Section \ref{sec.tightness}. The convergence of the number of zeros is then treated in Section \ref{Cvzeros}.
In order to facilitate the reading of the paper, the proof of some technical lemmas used in Section \ref{sec.marginals} are postponed in Appendix A.

\section{Finite dimensional convergence to a universal limit process} \label{sec.marginals}

As described above, the proofs of convergence in distribution for stochastic processes classically split into establishing the convergence of finite dimensional marginals and then establishing some tightness criteria. This section is devoted to the first task. 
More precisely, the next Section \ref{sec.kernel} establishes some technical results concerning the uniform convergence of covariance kernels, that will allow us to establish the convergence of the finite dimensional marginals of both processes $(X_n(t))$ and $(X_n'(t))$ and make explicit the covariance functions of their limits in Section \ref{sec.marginal}. The non-degeneracy and regularity of the limit process are then studied in Sections \ref{sec.nondeg} and \ref{sec.regu}.

\subsection{Qualitative and quantitatives kernel estimates} \label{sec.kernel} 
This Section is devoted to both qualitative and quantitative estimates for Riemann/ergodic sums involving periodic functions. Let us first introduce the function $K : \mathbb R \to \mathbb C$ defined by
\[
K(x):=e^{i x/2} \sin_c(x/2) =  \int_0^1 e^{i u x}du.
\]
The function $K$ is twice differentiable and one can easily check that the function and its derivatives are uniformly bounded. The next technical proposition will allow us to establish the convergence of the covariance function of the process $X_n(t)$ and to explicit the covariance of the limit. 
To facilitate the reading of the paper, its proof is postponed to Section \ref{proof.prop.conv} of the Appendix.

\begin{prop}\label{pro.conv.kernel}
Let $p_n$ be a sequence of real numbers such that the ratio $p_n/n$ converges to $\alpha \in (0,2\pi) \backslash \pi \mathbb Q$. 
Let $g$ and $h$ be two continuous $2\pi-$periodic functions which are piecewise $C^1$. Then, the sequences
\[
\begin{array}{ll}
C_n(s,t):= & \displaystyle{\frac{1}{n} \sum_{k=1}^n  g\left( \frac{k(p_n+t)}{n} \right)h\left( \frac{k(p_n+s)}{n} \right),}\\
\\
D_n(s,t):= & \displaystyle{\frac{1}{n} \sum_{k=1}^n \frac{k^2}{n^2}  g'\left( \frac{k(p_n+t)}{n} \right)h'\left( \frac{k(p_n+s)}{n} \right),}\\
\\
E_n(s,t):= & \displaystyle{\frac{1}{n} \sum_{k=1}^n \frac{k}{n}  g\left( \frac{k(p_n+t)}{n} \right)h'\left( \frac{k(p_n+s)}{n} \right),}
\end{array}
\]
converge uniformly as $n$ goes to infinity, namely we have
\begin{eqnarray}
\limsup_{n \to +\infty} \sup_{s,t \in [a,b]} \left| C_n(s,t) -\sum_{p\in \mathbb Z}\hat{g}(p) \hat{h}(-p) K(p(t-s)) \right|=0, \label{eq.conv1}\\ \nonumber
\\ 
\limsup_{n \to +\infty} \sup_{s,t \in [a,b]} \left| D_n(s,t) +\sum_{p\in \mathbb Z}\hat{g}(p) \hat{h}(-p) p^2 K''(p(t-s)) \right|=0, \label{eq.conv2}\\ \nonumber
\\ 
\limsup_{n \to +\infty} \sup_{s,t \in [a,b]} \left| E_n(s,t) +\sum_{p\in \mathbb Z}\hat{g}(p) \hat{h}(-p) p K'(p(t-s)) \right|=0. \label{eq.conv3}
\end{eqnarray}
Moreover, if $(\alpha, p_n)$ satisfy condition $(\star)$, there exists constants $D,\delta>0$ such that 
\begin{eqnarray}
  \sup_{s,t \in [a,b]} \left| D_n(s,t) +\sum_{p\in \mathbb Z}\hat{g}(p) \hat{h}(-p) p^2 K''(p(t-s)) \right| \leq \frac{D}{n^{\delta}} \label{eq.conv4}.
\end{eqnarray}
\end{prop}

\begin{rmk}The three series appearing in the last proposition are normally convergent. Indeed, if $g$ and $h$ are continuous and piecewise $C^1$, we have $||g||_2 <+\infty$, $||h||_2 <+\infty$ and $||g'||_2 <+\infty$, $||h'||_2 <+\infty$. Now, since $K$ is uniformly bounded, there exists a positive constant $c$ such that 
\[
\left|  \sum_{p\in \mathbb Z}\hat{g}(p) \hat{h}(-p) K(p(t-s))\right|^2 \leq c ||g||_2^2 ||h||_2^2,
\]
and we have similar upper bounds for the two other series since $K'$ and $K''$ are also uniformly bounded. 
\if{
\[
\left|  \sum_{p\in \mathbb Z}\hat{g}(p) \hat{h}(-p) p^2 K''(p(t-s))\right|^2 \leq  c' ||g'||_2^2 ||h'||_2^2,
\]
\[
\left|  \sum_{p\in \mathbb Z}\hat{g}(p) \hat{h}(-p) p K'(p(t-s))\right|^2 \leq  c'' ||g||_2^2 ||h'||_2^2.
\]
}\fi
\end{rmk}

\subsection{Convergence of the marginals} \label{sec.marginal} 
With the help of Proposition \ref{pro.conv.kernel}, we are now in position to establish the convergence of the process $(X_n(t),X_n'(t))_{t \in [a,b]}$ in the sense of finite marginals. 
 
\begin{prop}\label{pro.conv.marginal}
Let $(a_k)_{k \geq 1}$ be a sequence of independent and identically distributed random variables, which are centered  with unit variance. Let $f$ be a continuous and $2\pi-$periodic function, which is piecewise $C^1$. Let $p_n$ be a sequence of real numbers such that $p_n/n$ converges to $\alpha \in (0,2\pi) \backslash \pi\mathbb Q$. Then the stochastic process $(X_n(t),X_n'(t))_{t \in [a,b]}$ defined by Equation \eqref{eq.defX}
converges in the sense of finite marginals to a stationary and centered Gaussian process $(X_{\infty}(t), X_{\infty}'(t))_{t \in [a,b]}$ such that, for all $s,t$ 
\[
\mathbb E\left[ X_{\infty}(s) X_{\infty}(t) \right]   =: \rho(t-s), \quad \text{with} \quad \left \lbrace \begin{array}{l} \displaystyle{\rho(u):= \frac{1}{u} \int_{0}^u (f \ast \check{f})(x) dx}, \;\; u\neq 0,\\
\\
\rho(0)=\langle f, f\rangle.
\end{array}
\right.
\]
Moreover, for $0\leq i,j\leq 1$, the convergence of the covariance functions is uniform
\begin{equation}\label{eq.con.uni}
\lim_{n \to +\infty} \sup_{(s,t)\in [a,b]^2} \left| \mathbb E\left[ X_{n}^{(i)}(s) X_{n}^{(j)}(t) \right]- \mathbb E\left[ X_{\infty}^{(i)}(s) X_{\infty}^{(j)}(t) \right] \right|  =0,
\end{equation}
and, if $(\alpha, p_n)$ satisfy condition $(\star)$, then the convergence rate is polynomial, namely there exists $\kappa>0$ such that 
\begin{equation}\label{eq.con.poly}
\sup_{(s,t)\in [a,b]^2} \left| \mathbb E\left[ X_{n}^{(i)}(s) X_{n}^{(j)}(t) \right]- \mathbb E\left[ X_{\infty}^{(i)}(s) X_{\infty}^{(j)}(t) \right] \right| = O \left(\frac{1}{n^{\kappa}}\right).
\end{equation}
\end{prop}

\begin{rmk}
In the case of trigonometric polynomials corresponding to the choice $f(t)=\cos(t)$, one recovers the classical $\sin_c$ limit covariance. Indeed, we have then
\[
f\ast \check{f}(x) = \frac{1}{2\pi} \int_0^{2\pi} \cos(u)\cos(u-x) du = \frac{1}{2} \cos( x),
\]
so that for $u \neq 0$
\[
\rho(u) = \frac{1}{u} \int_{0}^u (f \ast \check{f})(x) dx = \frac{1}{2} \frac{\sin(u)}{ u}.
\]
\end{rmk}

\begin{proof}
For any $p-$tuple $(t_1, \ldots t_p)$, the convergence in law of the $2p-$dimensional random vector $(X_n(t_1), \ldots, X_n(t_p), X_n'(t_1), \ldots, X_n'(t_p))$ towards a centered Gaussian vector $(X_{\infty}(t_1), \ldots, X_{\infty}(t_p),X_{\infty}'(t_1), \ldots, X_{\infty}'(t_p))$ is an immediate consequence of the Central Limit Theorem for independent random variables. We are thus left to identify the limit covariance function. For $s,t \in [a,b]$, we have 
\[
\rho_n(s,t):=\mathbb E[ X_n(s)X_n(t)] = \frac{1}{n} \sum_{k=1}^n   f\left(\frac{k p_n}{n} +\frac{k s}{n}  \right)f\left(\frac{k p_n}{n} +\frac{k t}{n}  \right).
\]
The uniform convergence estimate \eqref{eq.conv1} of Proposition \ref{pro.conv.kernel}, applied  with $g=h=f$ then yields  that
\[
\begin{array}{ll}
\displaystyle{\lim_{n\to +\infty} \rho_n(s,t)}=\displaystyle{ \sum_{p\in \mathbb Z} \hat{f}(p) \hat{f}(-p) K(p(t-s)):=\rho(t-s),} 
\end{array}
\]
where 
\[
\begin{array}{ll}
\rho(u) & := \displaystyle{\sum_{p \in \mathbb Z}  \hat{f}(p) \hat{f}(-p) e^{i p  u/2}  \sin_c \left(\frac{p u}{2}\right)=\sum_{p \in \mathbb Z}    \hat{f}(p) \hat{f}(-p)   \times  \int_{0}^1 e^{ i p u  x}dx} \\
\\
& = \displaystyle{\frac{1}{u} \int_{0}^u \left( \sum_{p \in \mathbb Z}    \hat{f}(p) \hat{f}(-p)  \times e^{ i p x}\right)dx}.
\end{array}
\]
Recalling that for general periodic functions $h$ and $g$, the product $\hat{g}(p) \hat{h}(p)$ is the Fourier coefficient of order $p$ of the convolution $g \ast h$, we get that the limit covariance function is given by
\[
\rho(u) = \frac{1}{u} \int_{0}^u (f \ast \check{f})(x) dx.
\]
Let us now focus on the convergence of the covariances $\mathbb E[X_n'(s)X_n(t)]$. We have
\[
\mathbb E\left[ X_{n}'(s) X_{n}(t) \right]  =  \frac{1}{n} \sum_{k=1}^n \frac{k}{n}  {f}^{}\left( \frac{k(p_n+t)}{n} \right){f'}^{}\left( \frac{k(p_n+s)}{n} \right),
\]
so that the result follows immediately from the estimate \eqref{eq.conv3} applied to $g=f$ and $h=f'$. Last, the covariance of the derivatives is given by 
\[
\mathbb E\left[ X_{n}'(s) X_{n}'(t) \right]  =  \frac{1}{n} \sum_{k=1}^n \frac{k^2}{n^2}  {f}^{'}\left( \frac{k(p_n+t)}{n} \right){f'}^{}\left( \frac{k(p_n+s)}{n} \right),
\]
and the conclusion follows this time from the uniform estimate \eqref{eq.conv2}. The quantitative estimate \eqref{eq.con.poly} is then naturally a direct consequence of the kernel estimate \eqref{eq.conv4}.
\end{proof}

In the following, we will need a lower bound on the covariance kernel of the derivative process $(X_{\infty}'(t))_{t \in [a,b]}$. This is object of the next lemma.

\begin{lma}\label{lem.minorcovar}
The covariance of the process $(X_{\infty}'(t))_{t \in [a,b]}$ is given by 
\[
 \rho''(t)-\rho''(0) = \frac{1}{t^3} \int_0^t \frac{s^2}{2} \left( f'\ast \check{f'}(0) - f' \ast \check{f'}(s)\right).
\]
In particular, for any small $\delta>0$, we have 
\[
\inf_{|t-s|\geq \delta} \mathbb E[|X_{\infty}'(t)-X_{\infty}'(s)|^2] =2 \inf_{t \geq \delta} (\rho''(t)-\rho''(0))>0.
\]
\end{lma}

\begin{proof}
By definition of the covariance function $\rho(t)$, we have 
 \[
 t \,\rho(t) = \int_0^t f \ast \check{f}(x)dx,
 \]
 so that 
$
 \rho(t) + t \rho'(t) = f\ast \check{f} (t)$ and $2 \rho'(t) + t \rho''(t) = -f\ast \check{f'}(t)$.
 In particular, since $\rho'(0)=0$ and $f\ast \check{f'}(0)=0$, we get that for all $t$
 \[
 2 (\rho'(t) -\rho'(0))+ t \rho''(t) =  f\ast \check{f'}(0)-f\ast \check{f'}(t).
 \]
Subtracting $3\rho''(0)=-f' \ast f'(0)=-||f'||_2^2$ on each side of the equation and dividing by $t\neq 0$, we get
 \[
 \begin{array}{ll}
\displaystyle{ 2 \left(  \frac{1}{t} \int_0^t (\rho''(s)-\rho''(0)) ds\right)  + \rho''(t)-\rho''(0)} & = \displaystyle{ \frac{1}{t} \left(   f\ast \check{f'}(0) -  f\ast \check{f'}(t)\right)  + f' \ast \check{f'}(0)} \\
 \\
 &= \displaystyle{\frac{1}{t} \int_0^t    \left( f'\ast \check{f'}(0) - f' \ast \check{f'}(s)\right)ds. }
 \end{array}
 \]
Let us introduce the function $g$ defined as
\[
g(t)= \int_0^t [\rho''(s)-\rho''(0)] ds, \quad i.e.\quad  g'(t) = [\rho''(t)-\rho''(0)].
\]
For all $t \neq 0$, the last equality thus reads
\begin{equation}\label{eq.comparg}
\frac{2}{t} g(t) + g'(t) = \frac{1}{t} \int_0^t    \left( f'\ast \check{f'}(0) - f' \ast \check{f'}(s)\right)ds
\end{equation}
or equivalently
\[
2t g(t) + t^ 2 g'(t) = t \int_0^t    \left( f'\ast \check{f'}(0) - f' \ast \check{f'}(s)\right)ds.
\]
Therefore 
\[
[t^2 g(t)]' = t \int_0^t    \left( f'\ast \check{f'}(0) - f' \ast \check{f'}(s)\right)ds,
\]
and an integration by parts gives 
\[
\begin{array}{ll}
t^2 g(t) & = \displaystyle{\int_0^t  \left( s \int_0^s    \left( f'\ast \check{f'}(0) - f' \ast \check{f'}(u)\right)du \right)ds}\\
\\
 &= \displaystyle{\frac{t^2}{2} \int_0^t    \left( f'\ast \check{f'}(0) - f' \ast \check{f'}(s)\right)ds - \int_0^t \frac{s^2}{2} \left( f'\ast \check{f'}(0) - f' \ast \check{f'}(s)\right)ds.}\\
\end{array}
\]
In other words, we have 
\[
\frac{2}{t} g(t) - \frac{1}{t} \int_0^t    \left( f'\ast \check{f'}(0) - f' \ast \check{f'}(s)\right)ds=- \frac{1}{t^3} \int_0^t \frac{s^2}{2} \left( f'\ast \check{f'}(0) - f' \ast \check{f'}(s)\right)ds.
\]
A comparison with Equation \eqref{eq.comparg} then yields the identification
\begin{equation}\label{eq.min.covar}
g'(t)=  \rho''(t)-\rho''(0) = \frac{1}{t^3} \int_0^t \frac{s^2}{2} \left( f'\ast \check{f'}(0) - f' \ast \check{f'}(s)\right)ds.
\end{equation}
Now, going back to Fourier series, via Parseval equality, we have for all $s \notin \pi \mathbb Z$
\[
f'\ast \check{f'}(0) -f' \ast \check{f'}(s) = \sum_{p \in \mathbb Z} |\widehat{f'}(p)|^2 \left (1-e^{ips}\right) =  \sum_{p \in \mathbb Z} |\widehat{f'}(p)|^2 \left (1-\cos(ps)\right)>0,
\]
so that by Equation \eqref{eq.min.covar}, we deduce that $\rho''(t)-\rho''(0)>0$ for all $t>0$. By continuity, for all small $\delta>0$, we have then  
\[
\inf_{\delta \leq t \leq 1} (\rho''(t)-\rho''(0))>0.
\]
\end{proof}

\subsection{Regularity of the limit process} \label{sec.regu}
Let us now establish that, as soon as $f$ is piecewise $C^{1+\alpha}$ for $\alpha>0$, the Gaussian limit process $X_{\infty}$ admits a $C^1$ modification, i.e. $X_{\infty}'$ admits a continuous modification. We emphasize here that the result is not obvious since, if $f$ in only piecewise $C^{1+\alpha}$, then for any fixed $n$, the process $X_n'$ may have jumps.  
\begin{lma}\label{lem.conti}
If the function $f$ is continuous and piecewise $C^{1+\alpha}$ for $\alpha>0$, there exists a finite constant $C$ such that the limit Gaussian process $(X_{\infty}'(t))_{t \in [a,b]}$ satisfies the inequality 
\begin{equation}\label{eq.kolmogo}
 \mathbb E\left[ |X_{\infty}'(s)-X_{\infty}'(t)|^2 \right] \leq C |t-s|^{\alpha}.
\end{equation}
In particular, the process admits a continuous modification and we have 
\[
 \mathbb E\left[ \sup_{t \in [a,b]} |X_{\infty}'(t)| \right]  <+\infty.
\]
\end{lma}

\begin{proof}
By Proposition \ref{pro.conv.marginal}, we have 
\[
\begin{array}{ll}
 \mathbb E\left[ |X_{\infty}'(s)-X_{\infty}'(t)|^2 \right] & \displaystyle{= -2[ \rho''(0)-\rho''(t-s)]} .
 \end{array}
 \]
By Lemma \ref{lem.minorcovar}, the regularity at the origin of $\rho''(t)$ is the same as the one of $f'\ast \check{f'}(t)$. Now, if $f'$ is piecewise $C^\alpha$ then the convolution product $f'\ast \check{f'} $ is continuous and $C^{\alpha}$, hence there exist a finite positive constant $C$ such that Equation \eqref{eq.kolmogo} holds. 
Using the Gaussian hypercontractivity, we then deduce that for $p>0$ large enough, and for new constants $C'>0$ and $\alpha'>1$, we have
\[
 \mathbb E\left[ |X_{\infty}(s)-X_{\infty}(t)|^{2p} \right] \leq C' |t-s|^{\alpha'}.
\]
The celebrated Kolmogorov regularity criterion thus ensures that the limit process $(X_{\infty}'(t))$ admits a modification which is almost surely continuous, see e.g. Theorem 1.8 p. 19 of \cite{yor}. Since our process is indexed by the compact set $[a,b]$, the almost sure continuity of the sample paths implies their boundedness. In a Gaussian context, Borell's inequality shows that this sole boundedness is sufficient to ensure the integrability  of the supremum, as established e.g. in Theorem 2.1 p. 43 of \cite{adler}.
\end{proof}

\subsection{Non-degeneracy of the limit process} \label{sec.nondeg}
Let us now establish that the limit process $X_{\infty}$ in non-degenerate in the sense that almost surely, it does not exhibit double zeros. 

\begin{lma}\label{lem.nondeg}
The limit process $(X_{\infty}(t))_{t \in [a,b]}$ is non degenerate in the sense that almost surely, we have $X_{\infty}'(t)\neq 0$ if $X_{\infty}(t)= 0$.
\end{lma}

\begin{proof}
The non-degeneracy of the limit process is a direct consequence of Bulinskaya Lemma, see e.g. Proposition 6.11 of \cite{azaisW}. We are left to check that the Gaussian vector 
$(X_{\infty}(t) , X_{\infty}'(t) )$ admits a uniformly bounded density on any compact time interval $[a,b]$. But the process being stationary, the law of $(X_{\infty}(t) , X_{\infty}'(t) )$ does not depend on $t$ and is a centered Gaussian vector with covariance 
\[
\left(
\begin{array}{cc}
\rho(0) & 0 \\
0 & -\rho''(0)
\end{array}\right) = \left(
\begin{array}{cc}
\langle f, f \rangle & 0 \\
0 & \frac{1}{3} \langle f', f' \rangle
\end{array}\right) .
\]
Remembering the hypotheses on $f$ and $f'$ at the very beginning of Section \ref{sec.state}, we get that the determinant of this matrix is positive, hence the result.
\end{proof}

\section{Tighness estimates depending on the regularity of $f$} \label{sec.tightness}
In this section we prove that under each of the additional conditions $i),ii),iii)$ or $iv)$ appearing in the statement of Theorem \ref{theo.C2}, the laws of $(X_n(t))_{t \in [a,b]}$ are tight with respect to the $C^1([a,b])$ topology. 
\subsection{Tightness in the regular cases}

Following Theorem 1 and Remark 1 of \cite{rusakov}, sufficient conditions ensuring the convergence in the $C^1$ topology on $[a,b]$ are
on the one hand the convergence of the finite dimensional marginals, and on the other hand, a control on the increments of both the process $X_n(t)$ and its first derivative. 
The convergence of the finite dimensional marginals  having been established in the last section, we are left to check the tightness criterion concerning the increments. 

\begin{prop}\label{prop.smooth}
Under conditions $i)$ or $ii)$ of Theorem \ref{theo.C2}, the laws of $(X_n(t))_{t \in [a,b]}$ are tight with respect to the $C^1([a,b])$ topology.
\end{prop}
\begin{proof}
Let us first consider condition $i)$, i.e. the case where the function $f$ of class $C^2$. We have then $||f'||_{\infty}<+\infty$ and $||f'||_{\infty}<+\infty$ and a direct computation yields
\begin{equation} \label{eq.tension}
\mathbb E[|X_n(t)-X_n(s)|^2] \leq ||f'||_{\infty} |t-s|^2, \qquad \mathbb E[|X_n'(t)-X_n'(s)|^2] \leq ||f''||_{\infty} |t-s|^2, 
\end{equation}
hence the result. In the case $ii)$ where $f$ is only $C^{1+\alpha}$, the first estimate of Equation \eqref{eq.tension} is still valid. Otherwise, if $\mathbb E\left[|a_1|^{\beta}\right]<+\infty$ for some $\beta>1/\alpha$, for fixed constant $\lambda, u, v \in \mathbb R$, we can consider the sequence $(M_n)_{n \geq 0}$ defined by $M_0=0$ and for $n \geq 1$
\[
M_n:=\lambda \sum_{k=1}^n a_k k [f'(k u) -f'(kv)],
\]
which is a martingale with respect to the filtration $\mathcal F_n:=\sigma(a_1, \ldots, a_n)$. By the Burkholder--Davis--Gundy inequality, there exists a universal constant $C_{\beta} $ which only depends on $\beta$ such that
\begin{equation}\label{eq.bdg}
\mathbb E[|M_n|^{\beta}] \leq C_{\beta} \mathbb E[|D_n|^{\beta/2}],
\end{equation}
where
\[
D_n := \sum_{k=1}^n (M_k-M_{k-1})^2 = \lambda^2 \sum_{k=1}^n a_k^2 k^2 [f'(k u) -f'(kv)]^2.
\]
Setting $C_{\alpha}:= \sup_{s \neq t} |f'(t) - f'(s)| / |t-s|^{\alpha}$, by H\"older inequality, we have 
\[
D_n \leq \lambda^2 C_{\alpha}^2 \sum_{k=1}^n a_k^2 k^{2(1+\alpha)} |u -v|^{2\alpha} \leq \lambda^2 C_{\alpha}^2 n^{3+2\alpha}| u -v|^{2\alpha} \times \left|\frac{1}{n} \sum_{k=1}^n |a_k|^{\beta}\right|^{2/\beta}.
\]
In particular, one gets that for all $\lambda, u, v$
\[
\mathbb E[|D_n|^{\beta/2}] \leq  \lambda^{\beta} C_{\alpha}^{\beta} \mathbb E[|a_1|^{\beta}] n^{3\beta/2 +\alpha \beta}| u -v|^{\alpha \beta},
\]
which, in combination with Equation \eqref{eq.bdg}, yields
\[
\mathbb E[|M_n|^{\beta}] \leq C_{\beta} \lambda^{\beta} C_{\alpha}^{\beta} \mathbb E[|a_1|^{\beta}] n^{3\beta/2 +\alpha \beta}| u -v|^{\alpha \beta}.
\]
Setting $u=s/n$, $v=t/n$, $\lambda=n^{-3/2}$, one finally gets that for all $n, t,s$
\[
\mathbb E[|X_n'(s) - X_n'(t)|^{\beta}] \leq C_{\beta} C_{\alpha}^{\beta} \mathbb E[|a_1|^{\beta}] | t -s|^{\alpha \beta},
\]
hence the result.

\end{proof}
\subsection{Tightness under the fourth moment assumption}

The na\"{i}ve approach described above seems to highlight the necessity of having sufficiently high moments to compensate low H\"{o}lder exponents. Nevertheless, somehow surprisingly, the next Theorem ensures that a fourth moment is always enough, whatever is the H\"{o}lder exponent, in order to guarantee tightness in the $C^1$ topology.

\begin{prop}\label{Thm/C1/moment4}
Suppose that condition $iii)$ of Theorem \ref{theo.C2} holds, i.e. that $f$ is $C^{1+\alpha}$ for some $\alpha>0$ and that $\mathbb E[a_1^4]<\infty$. Then the laws of the processes $(X_n'(t))_{t\in [a,b], n \geq 1}$ are tight with respect to the uniform topology on $[a,b]$. As a consequence, the laws of the processes $(X_n(t))_{t\in [a,b], n \geq 1}$ are tight with respect to the $C^1$ topology on $[a,b]$.
\end{prop}

\begin{proof}
Let us first note that, when working in the space of continuous functions on $[a,b]$ endowed with the topology associated with the uniform norm, the operation of taking the anti-derivative is a continuous map. Therefore, to prove Theorem \ref{Thm/C1/moment4}, it is sufficient to establish that the laws of the derivatives $(X_n'(t))_{t\in [a,b], n \geq 1}$ are tight for the uniform topology. 
Our goal is thus to establish that, for any $\lambda>0$ and $[a,b] \subset \mathbb R$
\begin{equation}\label{tightnesscriterion}
\lim_{\delta\to 0}\limsup_{n\to\infty}~~\mathbb{P}\left(\sup_{
\substack{
(t,s)\in [a,b]^2\\
|t-s|\le \delta}}
 \Big{|}X_n'(t)-X_n'(s)\Big{| }> \lambda\right)=0,
\end{equation}
where we recall that
\[
X_n'(t)=\frac{1}{\sqrt{n}}\sum_{k=1} a_k \frac{k}{n} f'\left(\frac{k(p_n+t)}{n}\right).
\]
The proof follows several steps that we detail below. \par
\medskip
\noindent
\underline{{\bf Step 1:} discretizing the maximum.}\par
\medskip
\noindent

In order to establish the criterion (\ref{tightnesscriterion}) we first discretize the maximum. Simply fix an integer $p>0$, $\eta=\frac{b-a}{p}$ so that

$$[a,b]=\bigcup_{k=0}^{p-1}[a+k \eta,a+(k+1)\eta].$$

For any $t\in[a,b]$ we can find $l_t\in \llbracket 0, p-1 \rrbracket$ such that $t\in [a+l_t \eta,a+(l_t+1)\eta]$ and we then set $x_t:=a+l \eta$. In this case, relying on the Hölder property of $f$ we can write

\begin{eqnarray*}
|X_n'(t)-X_n'(x_t)| &= &\left|\frac{1}{\sqrt{n}}\sum_{k=1}^n a_k \frac{k}{n} \left(f'\left(\frac{k(p_n+t)}{n}\right)-f'\left(\frac{k(p_n+x_t)}{n}\right)\right)\right|\\
&\le& C \frac{\eta^{\alpha}}{\sqrt{n}} \sum_{k=1}^n |a_k| = C (b-a)^\alpha \frac{\sqrt{n}}{p^\alpha} \left(\frac{1}{n}\sum_{k=1}^n |a_k|\right).
\end{eqnarray*}
So if $p=p_n=n^{\frac{1}{2\alpha}+\epsilon}$, for some fixed $\epsilon >0$, then
$\mathbb E\left( \sup_{t\in[a,b]} |X_n'(t)-X'(x_t)|\right)\to 0.$\par
\medskip
As a result, one is left to establish the corresponding tightness criterion for the process $(X'(x_t))_{t\in[a,b]}$, namely

\begin{equation*}
\lim_{\delta\to 0}\limsup_{n\to\infty}~~\mathbb{P}\left(\sup_{
\substack{
(t,s)\in[a,b]^2\\
|t-s|\le \delta}}
 \Big{|}X_n'(x_t)-X_n'(x_s)\Big{|}>\lambda \right)=0.
\end{equation*}

However, since $x_t$ takes values in the finite set $D_n:=\bigcup_{k=0}^p \{a+k \eta\}$, it simply remains to establish that

\begin{equation}\label{tightnesscriterion-disc}
\lim_{\delta\to 0}\limsup_{n\to\infty}~~\mathbb{P}\left(\sup_{
\substack{
(u,v)\in D_n^2\\
|u-v|\le \delta}}
 \Big{|}X_n'(u)-X_n'(v)\Big{|}> \lambda\right)=0.
\end{equation}

\newpage
\noindent
\underline{{\bf Step 2:} imposing a lower distance in the discrete maximum.}

\medskip

Fix $\delta>0$. Observe that for any $(u,v)\in D_n^2$ such that $|u-v|\le \delta$, there exist $w \in D_n$ such that $\delta/3 \le |u-w|\le \delta$ and $\delta/3 \le |v-w|\le \delta$. As a result, we simply have

$$\sup_{
\substack{
(u,v)\in D_n^2\\
|u-v|\le \delta
}}
 \Big{|}X_n'(u)-X_n'(v)\Big{|}\le 2 \sup_{
\substack{
(u,v)\in D_n^2\\
\delta/3 \le|u-v|\le \delta
}}
\Big{|}X_n'(u)-X_n'(v)\Big{|},$$
and one is finally left to show that for any $\lambda>0$ we have
\begin{equation}\label{tightnesscriterion-disc-bis}
\lim_{\delta\to 0}\limsup_{n\to\infty}~~\mathbb{P}\left(\sup_{
\substack{
(u,v)\in D_n^2\\
\delta/3 \le |u-v|\le \delta
}}
 \Big{|}X_n'(u)-X_n'(v)\Big{|}> \lambda\right)=0.
\end{equation}

\medskip
\noindent
\underline{{\bf Step 3:} using a high dimensional Berry-Essen bound.}
\medskip

The next step consists in showing that the case of general random coefficients reduces to the Gaussian case. For this, we will use an appropriate, high dimensional, Berry--Esseen type argument, as recently developed in by V. Chernozhukov, D. Chetverikov and K. Kato in a nice series of papers. Namely, we will use Proposition 2.1 from \cite{chernozhukov2017central}, in a slightly  weaker form. Such a result indeed allows to deal with uniform estimates for sums of independent variables, in particular with estimates like \eqref{tightnesscriterion-disc-bis}, as soon as the supremum is taken on set that grows polynomially with $n$.
In order to remain self-contained, we restate the statement here, modifying a little bit their notations to avoid confusion with ours. Let consider $Y_1,\ldots,Y_n$  some independent random vectors of $\mathbb R^p$ and for each $i\in\{1,\cdots,n\}$ we denote by $(Y_{i,j})_{1\le j \le p}$ their coordinates. 

\begin{thm}[Proposition 2.1 of \cite{chernozhukov2017central}] \label{theo.chernouille}
Assume that the following conditions are satisfied:

\begin{itemize}
\item[(A1)] for all $(i,j)\in\{1,\cdots,n\}\times \{1,\cdots,p\}$,  $Y_{i,j}\in L^4(\mathbb{P})$,
\medskip
\item[(A2)] there exists $b>0$ such that
\begin{equation*}
\forall n\ge1,\,\,\forall j\in \{1,\cdots,p\},\,\,\frac{1}{n}\sum_{i=1}^n \mathbb E\left(Y_{i,j}^2\right)\ge b,
\end{equation*}

\medskip
\item[(A3)] there exists a constant $C>0$ such that
\begin{equation*}
\forall n\ge1,\,\,\forall j\in \{1,\cdots,p\},\,\,\frac{1}{n}\sum_{i=1}^n \mathbb E\left(Y_{i,j}^4\right)\le C,
\end{equation*}
\medskip
\item[(A4)]
\begin{equation*}
\forall n\ge1,\,\,\forall i \in \{1,\cdots,n\},\,\mathbb E\left(\max_{1\le j \le p} \left|\frac{Y_{i,j}}{C}\right|^4\right)\le 2.
\end{equation*}
\end{itemize}
Then, introducing $Z_1,\ldots,Z_n$ a Gaussian vector with same covariance matrix as $Y_1,\ldots,Y_n$, setting $\Pi_p:=]-\infty,t_1]\times\cdots\times]-\infty,t_p]$, it holds that

\begin{eqnarray}\label{Berry-Essen-bound}
&\displaystyle{\sup_{(t_1,t_2,\cdots,t_p) \in \mathbb R^p}}\left|\mathbb{P}\left(\frac{{S_n^{Y}}}{\sqrt{n}}\in \,\Pi_p \right)-\mathbb{P}\left(\frac{{S_n^{Z}}}{\sqrt{n}}\in\,\, \Pi_p \right)\right|\le K\left( \frac{\log(p n)^{\frac{7}{6}}+\log(pn)}{n^{\frac{1}{6}}}\right),
\end{eqnarray}
where
\[
\frac{{S_n^{Y}}}{\sqrt{n}}:=\frac{1}{\sqrt{n}}\sum_{k=1}^n Y_k, \qquad
\frac{{S_n^{Z}}}{\sqrt{n}}:=\frac{1}{\sqrt{n}}\sum_{k=1}^n Z_k,
\]
and where the constant $K$ only depends on the parameters $C,b$ appearing in the assumptions. 
\end{thm}

To make more transparent the connection with Proposition 2.1 in \cite{chernozhukov2017central}, note that we have taken here $B_n$ to be constant and $q=4$. Hence, (A2) represents (M.1) from \cite{chernozhukov2017central} and (A3), (A4) represent (M.2) and (E.2) with $B_n=C$ and $q=4$.

In order to apply the last Berry--Esseen type estimate in our context, we make the following identification: $\forall (u,v) \in D_n^2\,\,\text{with}\,\,\delta/3\le |u-v|\le \delta$, we set
\[
Y_{i,(u,v)}:= a_i \left( f'\left(\frac{i}{n}(p_n+u)\right)-f'\left( \frac{i}{n}(p_n+v)\right)\right).
\]
Let us check that the above vectors meet all requirements for applying bound (\ref{Berry-Essen-bound}). We first observe that independence is straightforward as the inputs $(a_i)_{i\ge 1}$ are independent random variables. Besides, assumptions (A1)-(A3)-(A4) are together ensured by the fact that $a_1 \in L^4(\mathbb{P})$ and that $f'$ is continuous and $2\pi$-periodic, hence bounded. In order to establish (A2) we shall rely on the content of Section \ref{sec.marginal}. Namely, by Proposition \ref{pro.conv.marginal}, we know that
\begin{equation}\label{unif-conv}
\sup_{(t,s)\in[a,b]^2} \Big{|}\mathbb E\left( |X_n'(t)-X_n'(s)|^2\right)-\mathbb E\left( |X_\infty'(t)-X_\infty'(s)|^2\right)\Big{|}\xrightarrow[n\to\infty]~0,
\end{equation}
where $X_\infty'$ is a Gaussian process with covariance function $-\rho''$. 
Moreover, by Lemma \ref{lem.minorcovar}, we know that
\[
\min_{\delta/3\le |t-s|\le \delta}\mathbb E\left( |X_\infty'(t)-X_\infty'(s)|^2\right)\ge b_\delta>0.
\]
As result, for $n$ large enough, assumption (A2) is fulfilled for some suitable constant $b_\delta$ which only depends on $\delta$. Then, we may use the estimate (\ref{Berry-Essen-bound}). Since $\eta$ has been chosen of polynomial growth in $n$, it is clear that $\text{Card}(D_n^2)$ has also a polynomial size in $n$, hence the bound in (\ref{Berry-Essen-bound}) tends to zero as $n$ tends to infinity. Let us introduce $\widetilde{X}_n'^{}(\cdot)$ the process analogous to $X_n'(\cdot)$ where the random coefficients are simply replaced by independent standard random Gaussian variables. We can deduce that

\[
\mathbb{P}\left(\sup_{
\substack{
(u,v)\in D_n^2\\
\delta/3 \le |u-v|\le \delta}}
 \Big{|}X_n'(u)-X_n'(v)\Big{|}\ge \lambda\right)
-\mathbb{P}\left(\sup_{
\substack{
(u,v)\in D_n^2\\
\delta/3 \le |u-v|\le \delta}}
 \Big{|}\widetilde{X}'_n(u)-\widetilde{X}'_n(v)\Big{|}\ge \lambda\right)
\xrightarrow[n\to\infty]~0.
\]
As a conclusion, to establish the estimate \eqref{tightnesscriterion-disc-bis},  it is sufficient to treat the case where the random coefficients are i.i.d. standard Gaussian variables. 

\medskip
\noindent
\underline{{\bf Step 4:} the Gaussian case.}\par
\medskip
But in the case where the coefficients are Gaussian, we are back to the case treated in the proof of Proposition \ref{prop.smooth}. Indeed, since $f'$ is $C^{\alpha}$, we can use the Gaussian hypercontractivity to deduce that for $\beta>0$ large enough such that $\alpha \beta >1$, and for some constant $C>0$, we have 
\[
\mathbb E\left[| \widetilde{X}'_n(t)-\widetilde{X}'_n(s)|^{\beta}\right] \leq  C  | t-s|^{\alpha \beta}.
\]
In particular, by the standard Lamperti type criterion, the laws of $(\widetilde{X}'_n(t))$ are tight for the uniform topology. This ensures that for all $\lambda>0$, we have 
\[
\mathbb{P}\left(\sup_{
\substack{
(u,v)\in D_n^2\\
\delta/3 \le |u-v|\le \delta}}
 \Big{|}\widetilde{X}'_n(u)-\widetilde{X}'_n(v)\Big{|}\ge \lambda\right) \leq \mathbb{P}\left(\sup_{
 |u-v|\le \delta}
 \Big{|}\widetilde{X}'_n(u)-\widetilde{X}'_n(v)\Big{|}\ge \lambda\right)
\xrightarrow[n\to\infty]~0,
\]
hence the result.
\end{proof}

\subsection{Tighness in the piecewise linear case}\label{sec.tightaffine}
In this section, we establish that if the periodic function $f$ is continuous and piecewise linear, then the family of distributions of the associated sequence $(X_n'(t))_{t \in [a,b]}$ of random processes is tight for the standard Skorokhod topology on $D([a,b])$. Since we already know that the limit process $(X'_{\infty})_{t \in [a,b]}$ is continuous by Lemma \ref{lem.conti}, this implies in particular that the sequence $(X_n'(t))_{t \in [a,b]}$ is in fact tight for the uniform topology on $D([a,b])$, see \cite{skoro}. Hence the family of distributions of the sequence $(X_n(t))_{t \in [a,b]}$ is tight for the $C^1$ topology.
\par
\medskip
By Theorem 13.2 p. 139 of \cite{bilou}, the Skorokhod tightness criterion reduces here to the following control of the modulus of continuity, for all $\lambda>0$
\begin{equation}\label{tight-skoroo}
\lim_{\delta\to 0}\limsup_{n\to\infty} \mathbb{P}\left(\inf_{t=(t_i)} \max_{i}\sup_{t,s\in[t_{i-1},t_i[ }|X_n'(t)-X_n'(s)|\ge \lambda\right)=0.
\end{equation} 
where the infimum is taken on discretizations $t=(t_j)$ of $[a,b]$ such that $|t_{i}-t_{i -1}|\leq \delta$.
\begin{prop}\label{pro.skoroaff}
Suppose that condition $iv)$ of Theorem \ref{theo.C2} holds, i.e. that the periodic function $f'$ is piecewise constant and $\mathbb E[|a_1|^4]<+\infty$, then for all $\lambda>0$, we have
\begin{equation}\label{tight-skoro}
\lim_{\delta\to 0}\limsup_{n\to\infty} \mathbb{P}\left(\max_{0\leq i \leq \lfloor\frac{b-a}{\delta}\rfloor}\sup_{t,s\in[a+i \delta,a+ (i+1)\delta[}|X_n'(t)-X_n'(s)|\ge \lambda\right)=0.
\end{equation}
In particular, the sequence $(X_n'(t))_{t \in [a,b]}$ of random processes is tight for the standard Skorokhod topology on $D([a,b])$, and since the limit process $(X'_{\infty})_{t \in [a,b]}$, it is thus tight for the uniform topology.
\end{prop}

\begin{proof}
Let us first remark that in the statement of Proposition \ref{pro.skoroaff}, we choose the regular discretization where $t_i:=a + i \delta$ for $0 \leq i \leq \lfloor (b-a) /\delta\rfloor$, so that the estimate \eqref{tight-skoro} indeed implies the tightness criterion \eqref{tight-skoroo}.
The proof of Proposition \ref{pro.skoroaff} follows globally the same scheme as the one of Proposition \ref{Thm/C1/moment4}, but the last step i.e. the Gaussian case, requires some more work.  Let us first introduce few notations. By saying that $f$ is piecewise linear we mean that we can find $0=s_0<s_1<\cdots<s_{r+1}=1$ and $(\alpha_j,\beta_j)_{0\le j \le r}$ such that for all $x\in[0,1[$
\[
f(x)=\sum_{j=0}^r \textbf{1}_{[s_j,s_{j+1}[}(x) \left(\alpha_j x+\beta_j\right).
\]
In that case, we have naturally 
\[
X'_{n}(t)=\frac{1}{\sqrt{n}}\sum_{j=1}^r \alpha_j \sum_{k=1}^n a_k \frac{k}{n}   \mathds{1}_{[s_j,s_{j+1}[}\left( \frac{k}{n}\left(p_n+t\right)\right).
\]
\underline{\textbf{Step 1 :} the supremum is a maximum.}
\par
\medskip
First note that since the function $X_{n}'$ is piecewise constant, the supremum in Equation \eqref{tight-skoro} is a maximum on a finite set, namely we have
\[
\max_{0\leq i \leq \lfloor\frac{(b-a)}{\delta}\rfloor}\sup_{t,s\in[t_i, t_{i+1}[}|X_n'(t)-X_n'(s)| =\max_{t,s \in E_n^{\delta}}|X_n'(t)-X_n'(s)|,
\]
where the set $E_n^{\delta}$ is union of the set of points of discontinuity of the function $X_n'$ and the discretization points. Namely, the set $E_n^{\delta}$ is given by
\[
E_n^{\delta}:=\bigcup_{j=1}^r \bigcup_{i=0}^{\lfloor\frac{b-a}{\delta}\rfloor} \left( E_{n,i,j}^{\delta}\cup \{t_i, t_{i+1}\}\right),
\]
 with
\[
E^{\delta}_{n,i,j}:=\left\{t\in[t_i, t_{i+1}]\,\Big{|}\,\exists (k,q) \in \{1,\cdots,n\}\times \mathbb{Z},\,\frac{k}{n}\left(p_n+t\right)-2\pi q\in\{s_j,s_{j+1}\}\right\}.
\]

Note that we have then 
\begin{equation}\label{numbersingu}
\text{Card}\left( E^{\delta}_{n,i,j} \right) \le 2 n, \quad \text{and thus} \quad  \text{Card}\left( E^{\delta}_{n} \right) \leq 2 (n+1) r \left( \lfloor\frac{b-a}{\delta}\rfloor +1 \right).
\end{equation}
Indeed, let us argue by contradiction and suppose that $\text{Card}( E^{\delta}_{n,i,j}) \ge 1+2 n$. Then, the pigeonhole principle would ensure that there exists two disctincs elements $t,t'$ in $E^{\delta}_{n,i,j}$, two integers $(q,q') \in \mathbb{Z}$ and some integer $k\in\{1,\cdots,n\}$ such that
\[
\frac{k}{n}\left(p_n+t\right)-2\pi q= \frac{k}{n}\left(p_n+t'\right)-2\pi q'.
\]
Necessarily, we would have then $2\pi |q-q'|=\frac{k}{n}\left|t-t'\right|\le |t-t'|< \delta$, hence for $\delta$ small enough $q=q'$ and $t=t'$ which leads to a contradiction.

\par
\medskip
\noindent
\underline{\textbf{Step 2 :} lower bound on the distance between singularities.}
\par
\medskip 
As in the proof of Proposition \ref{Thm/C1/moment4}, using the triangle inequality, without loss of generality, we can moreover impose a lower bound on the distance between ``singularities''. Namely, to establish the estimate \eqref{tight-skoro}, it is sufficient to prove that 
\begin{equation}\label{Maxi-reduc2}
\lim_{\delta\to 0}\limsup_{n\to\infty}\mathbb{P}\left( \max_{\substack{t,s \in E^{\delta}_n \\ |t-s|\ge \frac{\delta}{3}}}|X_{n}'(t)-X_{n}'(s)|>\lambda\right)=0.
\end{equation}

\medskip
\noindent
\underline{{\bf Step 3:} using a high dimensional Berry-Essen bound.} \par
\medskip
As in the proof of Proposition \ref{Thm/C1/moment4}, we can now reduce the general case to the case of Gaussian coefficients thanks to Theorem \ref{theo.chernouille} recalled above.  
For convenience, we set this time
\[
D_n^{\delta} := \{ (t,s)\in E^{\delta}_n \times E^{\delta}_n , \; |t-s|\geq \delta/3\},
\]
and
\[
Y_i:=\left(Y_i(s,t)\right)_{(t,s) \in D_n^{\delta}}:=\left(a_i \times \frac{i}{n} \left(f'\left(\frac{i}{n}(p_n+t)\right)-f'\left(\frac{i}{n}(p_n+s)\right)\right)\right)_{(t,s) \in D_n^{\delta}}.
\]
Exactly as in the third step of the proof of Proposition \ref{Thm/C1/moment4}, one can check the vectors $Y_i$ satisfy the conditions (A1)--(A4) of Theorem \ref{theo.chernouille}. By the estimate \eqref{numbersingu}, we have moreover 
\[
\text{Card}(D_n^{\delta})\le 4 (n+1)^2 r^2 \left( \lfloor\frac{b-a}{\delta}\rfloor +1 \right)^2.
\]
Therefore, applying the bound \eqref{Berry-Essen-bound}, we get that
\[
\limsup_{n\to\infty}\left| \mathbb{P}\left( \max_{(t,s) \in D_n^{\delta}} |X_{n}'(t)-X_{n}'(s)|>\lambda\right)- \mathbb{P}\left(\max_{(t,s) \in D_n^{\delta}}|\widetilde{X}_{n}'(t)-\widetilde{X}_{n}'(s)|>\lambda\right)\right|= 0.
\]
Thus, we are left to show that
\begin{equation}\label{Maxi-reduc3bis}
\lim_{\delta\to 0}\limsup_{n\to\infty}\mathbb{P}\left( \max_{(t,s) \in D_n^{\delta}}|\widetilde{X}_{n}'(t)-\widetilde{X}_{n}'(s)|>\lambda\right)=0.
\end{equation}
\par
\medskip
\noindent
\underline{\textbf{Step 4 :} Maxima of high-dimensional Gaussian vectors.}\par
\medskip
\noindent
We shall use here Theorem 2 of \cite{chernouille2}, which we recall below.

\begin{thm} \label{theo.max}
Let $U=(U_j)_{1 \leq j \leq p}$ and $V=(V_j)_{1 \leq j \leq p}$ be two Gaussian vectors in $\mathbb R^p$ with covariances $\Sigma^U$ and $\Sigma^V$, then
\[
\sup_{x \in \mathbb R} \left|\mathbb P\left( \max_{1\leq j \leq p} U_j \leq x \right)-\mathbb P\left( \max_{1\leq j \leq p} V_j \leq x \right)\right| \leq C \times \Delta^{1/3} \left( 1 \vee \log\left( p/\Delta \right)\right)^{2/3}, 
\]
where $C$ is an universal constant that depends only on $\min_{1\leq j \leq p} \Sigma_{jj}^V$ and $\max_{1\leq j \leq p} \Sigma_{jj}^V$, and where 
\[
\Delta := \max_{1\leq i,j\leq p} \left|\Sigma_{ij}^V- \Sigma_{ij}^U\right|.
\]
\end{thm}
\noindent
Let us apply this result to the Gaussian vectors 
\[
U=\left(\widetilde{X}_{n}'(t)-\widetilde{X}_{n}'(s)\right)_{(t,s) \in D_n^{\delta}} \quad \text{and} \quad  V=\left({X}_{\infty}'(t)-{X}_{\infty}'(s)\right)_{(t,s) \in D_n^{\delta}}.
\]
In this case, we have $p=p_{n,\delta}=\text{Card}(D_n^{\delta})$ which grows polynomially in $n$ and thanks to the estimate \eqref{eq.con.poly}, under the condition $(\star)$,  there exists $\kappa >0$ such that 
\[
\begin{array}{lr}
\Delta & \displaystyle{=\max_{(t,s),(t',s') \in D_n^{\delta}} \left| \mathbb E \left[ \left(\widetilde{X}_{n}'(t)-\widetilde{X}_{n}'(s)\right)\left(\widetilde{X}_{n}'(t')-\widetilde{X}_{n}'(s')\right) \right] \right.} \\
\\
& \displaystyle{\left. - \mathbb E \left[ \left({X}_{\infty}'(t)-{X}_{\infty}'(s)\right)\left( X_{\infty}'(t')-{X}_{\infty}'(s')\right) \right] \right| = O\left(  \frac{1}{n^{\kappa}}\right)}.
\end{array}
\]
Moreover, by stationarity of the limit process, we have 
\[
 \mathbb E \left[ \left({X}_{\infty}'(t)-{X}_{\infty}'(s)\right)^2\right] =  2  \left| \rho''(0)-\rho''(t-s)\right|.
\]
By Lemma \ref{lem.minorcovar}, we know that there exists a positive constant $c_{\delta}>0$ such that 
\[
b_{\delta} \leq \min_{(t,s) \in D_n^{\delta} }  \mathbb E \left[ \left({X}_{\infty}'(t)-{X}_{\infty}'(s)\right)^2\right]. 
\]
Besides, by continuity of the correlation function at zero, there exists $C_{\delta}>0$ such that  
\[
\max_{(t,s) \in D_n^{\delta} }  \mathbb E \left[ \left({X}_{\infty}'(t)-{X}_{\infty}'(s)\right)^2\right] \leq C_{\delta}.
\]
Therefore, we can apply Theorem \ref{theo.max} to deduce that  
\[
\limsup_{n\to\infty}\left| \mathbb{P}\left( \max_{(t,s) \in D_n^{\delta}}|X_{\infty}'(t)-X_{\infty}'(s)|>\lambda\right)- \mathbb{P}\left( \max_{(t,s) \in D_n^{\delta}}|\widetilde{X}_{n}'(t)-\widetilde{X}_{n}'(s)|>\lambda\right)\right|= 0.
\]
To conclude, note that 
\[
\mathbb{P}\left( \max_{(t,s) \in D_n^{\delta}}|X_{\infty}'(t)-X_{\infty}'(s)|>\lambda\right) \leq \mathbb{P}\left( \sup_{\substack{t,s \in [a,b] \\ |t-s|\leq \delta}} |X_{\infty}'(t)-X_{\infty}'(s)|>\lambda\right)
\]
and since the limit process $X_{\infty}'$ is continuous, the standard criterion for tightness in the uniform topology ensures that the last upper bound goes to zero as $\delta$ goes to zero, hence the result.

\end{proof}


\section{Convergence in distribution of the number of zeros}\label{Cvzeros}
In this Section, we detail how the convergence results in the $C^1$ topology obtained above imply the convergence in distribution of the number of zeros of the underlying random functions. We first state and prove the following slight variation of Proposition 1 of \cite{rusakov}.
If $f$ is a continuous function on $[a,b]$, we set 
\[
\mathcal N(f,[a,b]):=\text{Card} \{ t \in [a,b], \; f(t)=0\} \in \mathbb{N}\cup\{\infty\}
\]
\begin{prop}
Let $(f_n)_{n \geq 1}$ be a sequence of continuous functions that are piecewise $C^1$ on $[a,b]$ and which converge to a $C^1$ function $f_{\infty}$ for the $C^1([a,b])$ topology :
\[
\lim_{n \to +\infty} \left( \sup_{x \in [a,b]} |f_n(x)-f_{\infty}(x)| + \sup_{x \in [a,b]} |f_n'(x)-f_{\infty}'(x)| \right) =0.
\]
Suppose moreover that $f_{\infty}$ is non-degenerate in the following sense: $f_{\infty}(a) f_{\infty}(b)\neq 0$ and $\omega(f_{\infty}):=\inf_{x \in [a,b]} |f_{\infty}(x)| + |f_{\infty}'(x)| >0$. Then, we have $\mathcal N(f_{\infty},[a,b])<+\infty$ and  
\[
\lim_{n \to +\infty} \mathcal N(f_n,[a,b]) =N(f_{\infty},[a,b]).
\]
\end{prop}

\begin{proof}
The condition $\omega(f_\infty)>0$ implies that there is no accumulation point of zeros of $f_\infty$ and that $\mathcal{N}(f_\infty,[a,b])<\infty$ follows. For any $\alpha<\min\left(\frac{\omega(f_\infty)}{2},|f(a)|,|f(b)|\right)$, the open set $f_\infty^{-1}\left(]-\alpha,\alpha[\right)$ is a finite union of open intervals containing exactly one zero of $f_\infty$. We may then write for $r=\mathcal{N}(f_\infty,[a,b])$ :
\[
f_\infty^{-1}\left(]-\alpha,\alpha[\right)=\bigcup_{i=1}^r ]x_i,y_i[,
\]
with $f_\infty(x_i)f_\infty(y_i)<0$. Besides, one each of these intervals, $|f_\infty'|>\frac{\omega(f_\infty)}{2}>0$ which implies that $f_\infty'>\frac{\omega(f_\infty)}{2}$  or else $f_\infty'<-\frac{\omega(f_\infty)}{2}$. Thus, for $n$ large enough, we may impose that
\begin{itemize}
\item[(i)] for all $i\in\{1,\cdots,r\}$, $f_n'>\frac{\omega(f_\infty)}{3}$ or $f_n'<-\frac{\omega(f_\infty)}{3}$ on $]x_i,y_i[$,

\medskip

\item[(ii)] for all $i\in\{1,\cdots,r\}$, $f_n(x_i)f_n(y_i)<0$,

\medskip

\item[(iii)] $|f_n|>0$ on $f_\infty^{-1}\left(]-\alpha,\alpha[\right)^{\text{c}}$.
\end{itemize}
Gathering properties (i), (ii) and (iii) and keeping in mind that $f_n$ is continuous implies that $\mathcal{N}(f_n,[a,b])=r=\mathcal{N}(f_\infty,[a,b])$ and proves the claim.
\end{proof}

\medskip

Relying on the continuous mapping Theorem on the metric space $E$ of continuous and piecewise $C^1$ functions on $[a,b]$ which is endowed with the canonical norm $\|f\|_E:=\sup_{x\in [a,b]}|f(x)|+|f'(x)|$, we have the following corollary.

\medskip

\begin{cor}
Let $X_n(\cdot)$ be a sequence of random processes taking values in the above metric space $E$ and converging in distribution towards $X_\infty(\cdot)$ which is assumed to be non-degenerate almost surely, we must have:

\[
\mathcal{N}\left(X_n(\cdot),[a,b]\right)\xrightarrow[n\to\infty]{\text{Law}}~\mathcal{N}\left(X_\infty(\cdot),[a,b]\right).
\]
\end{cor}
\bigskip
This last corollary associated with Theorem \ref{theo.C2} and Lemma \ref{lem.nondeg} naturally implies Theorem \ref{theo.ConvZero} stated in the introduction. Following the same lines as in \cite{iksanov2016}, one can then deduce the convergence of the point measure associated to the zeros as stated in Theorem \ref{theo.ConvPointMeas}.

\appendix
\label{appendix1}
\section{Some technical results} \label{proof.prop.conv}
In this section, we give the proof of Proposition \ref{pro.conv.kernel} stated in Section  \ref{sec.kernel} on the convergence of Riemann / ergodic sums, both in qualitative and quantitative ways.  
Let us first recall the definition of the functions $K$ 
\[
K(x):=e^{i x/2} \sin_c(x/2) =  \int_0^1 e^{i u x}du,
\]
and introduce its Riemann sum approximation $K_n$
\begin{equation}\label{eq.defKn}
K_n(x):=  \frac{1}{n} \sum_{k=1}^n\exp\left( \frac{i k x}{n}  \right)=e^{i \left( \frac{n+1}{2n}x\right)}\frac{ \sin \left(\frac{x}{2}\right) }{n \times \sin\left(\frac{ x}{2n}  \right) }.
\end{equation}
First note that both $K$ and $K_n$, and all their derivatives are then uniformly bounded and the first two derivatives are given by 
\[
K'(x) =  i \int_0^1 u e^{i u x}du, \qquad K''(x) =  - \int_0^1 u^2 e^{i u x}du,
\]
and
\begin{equation}\label{eq.defKprime}
K_n'(x)=  \frac{i}{n} \sum_{k=1}^n\frac{k}{n} \exp\left( \frac{i k x}{n}  \right), \quad K_n''(x) = -\frac{1}{n} \sum_{k=1}^n\frac{k^2}{n^2} \exp\left( \frac{i k x}{n}  \right).
\end{equation}
Moreover, direct computations show that there exists positive constants $C,C' C''$ such that for all $x \notin \pi \mathbb Z$, we have 
\begin{equation}\label{eq.KnKprimen}
|K_n(x)| \leq  \frac{1}{n \times \left| \sin\left(\frac{ x}{2n}  \right)\right|}, \quad |K_n'(x)| \leq  \frac{C}{n \times \left|\sin\left(\frac{ x}{2n}  \right)\right|}+\frac{C'}{n^2 \times \sin^2\left(\frac{ x}{2n}  \right)},
\end{equation}
and similarly 
\begin{equation}\label{eq.Kseconden}
 \quad |K_n''(x)| \leq \frac{C}{n \times \left|\sin\left(\frac{ x}{2n}  \right)\right|}+ \frac{C'}{n^2 \times \sin^2\left(\frac{ x}{2n}  \right)}+\frac{C''}{n^3 \times \left|\sin^3\left(\frac{ x}{2n}  \right)\right|}.
\end{equation}

\subsection{Some preliminary lemmas} 
Let us first state the following lemma dealing with the uniform asymptotic behavior of $K_n$ and its derivatives on compact sets.
\begin{lma}\label{lem.convuni}
If $A$ is a compact set of the real line, then for $0 \leq i \leq 2$
\begin{equation}\label{eq.unifK}
\lim_{n \to +\infty} \sup_{x \in A} |K_n^{(i)}(x)-K^{(i)}(x)| =0.
\end{equation}
\end{lma}
\begin{proof}
In Equations \eqref{eq.defKn} and \eqref{eq.defKprime} above, the functions $K_n$ and its derivatives appear as Riemann sums, which converge uniformly on compact sets, hence the result.
\end{proof}
Moreover, the function $K_n$ and its derivatives uniformly vanish at infinity in the following sense.
\begin{lma}\label{lem.convuni2}
Let $p_n$ be a sequence of real numbers such that the ratio $p_n/n$ converges to $\alpha \in (0,2\pi) \backslash \pi \mathbb Q$ and let $Q$ be a positive integer, then for $0 \leq i \leq 2$
\begin{eqnarray}
\label{eq.unifK}
\lim_{n \to +\infty} \sup_{s,t \in [a,b]} \sup_{\substack{p \neq -q \in \mathbb N \\ |p| \leq Q, |q| \leq Q}}
\left|K_n^{(i)}\left(  (p+q)p_n+ (pt+qs)\right)\right|=0.
\end{eqnarray}
\end{lma}

\begin{proof}
From Equations \eqref{eq.KnKprimen} and \eqref{eq.Kseconden} which give upper bounds on $K_n$ and its derivatives, it is sufficient to show that  $\sin((p+q)p_n/2n+ (pt+qs)/2n)$ is bounded away from zero. 
We have 
\[
\sin \left(\frac{ (p+q)p_n+ (pt+qs)}{2n}\right) = \sin \left(\frac{(p+q)\alpha}{2}+  \frac{(p+q)}{2} \left(\frac{p_n}{n}-\alpha  \right)+ \frac{(pt+qs)}{2n}\right)
\]
so that uniformly in $|p|,|q| \leq Q$, $p \neq -q$ and $s,t \in [a,b]$, 
\[
\sin \left(\frac{ (p+q)p_n+ (pt+qs)}{2n}\right)= \sin \left(\frac{(p+q)\alpha}{2} +o(1)\right).
\]
Since $\alpha \in (0,2\pi) \backslash \pi \mathbb Q$, we have uniformly in $|p|,|q| \leq Q$ with $p\neq -q$
\[
\left|\sin \left(\frac{(p+q)\alpha}{2}\right)\right|>0,
\]
hence the result. 
\end{proof}
Moreover,  if $\alpha/\pi$ is a Diophantine number, the last estimate can be quantified in the following way.
\begin{lma}\label{lem.convuni3}Suppose that $(\alpha,p_n)$ satisfies condition $(\star)$, in other words suppose that
there exists $c_\alpha>0$ and $\nu_\alpha \in ]0,1[$ such that for every $(p,q)\in\mathbb{N}\times \mathbb N^*$ 
\begin{equation}\label{Def-Dioph}
\left|\frac{\alpha}{\pi}-\frac{p}{q}\right|\ge \frac{c_\alpha}{q^{2+\nu_\alpha}}.
\end{equation}
and that $p_n/n$ converges to $\alpha$ at some polynomial speed, namely
\begin{equation}\label{conv-dioph}
\exists c_1,c_2>0,\,\,\text{s.t.}\,\,\left|\frac{p_n}{n}-\alpha\right|\le \frac{c_1 }{n^{c_2}}.
\end{equation}
Then for any $0<\gamma<1/2$, there exists a positive constant $C$ such that for $0 \leq j \leq 2$
\begin{equation}\label{eq.unifK21}
 \sup_{s,t \in [a,b]} \sum_{|p| \leq n^{\gamma}}
\left|K_n^{(j)}\left(  p(t-s)\right)-K^{(j)}\left(  p(t-s)\right)\right|\leq \frac{C}{n^{\delta}}, \;\; \text{with} \;\; \delta = 1-2\gamma.
\end{equation}
Moreover, for all $\gamma$ such that
\[
0<\gamma < \frac{1}{3+\nu_{\alpha}} \wedge \frac{c_2}{2+\nu_{\alpha}}
\]
there exists $C>0$ such that for $0 \leq j \leq 2$
\begin{equation}\label{eq.unifK2}
 \sup_{s,t \in [a,b]} \sum_{\substack{p \neq -q \in \mathbb N \\ |p| \leq n^{\gamma}, |q| \leq n^{\gamma}}}
\left|K_n^{(j)}\left(  (p+q)p_n+ (pt+qs)\right)\right|\leq \frac{C}{n^{\kappa}}, \;\; \text{with} \;\; \kappa := 1- \gamma(3+\nu_\alpha).
\end{equation}
\end{lma}

\begin{proof}
Let us focus on the first estimate (\ref{eq.unifK21}). Given some $f\in\mathcal{C}^1([0,1])$, we recall that
\[
\left|\frac{1}{n}\sum_{k=1}^n f\left(\frac{k}{n}\right)-\int_0^1 f(u) d
u\right| \le \frac{\|f'\|_\infty}{2 n}.
\]
For $0\leq j \leq 2$, applying the latter to $f(u)=(iu)^j e^{ i  x u}$ gives the bound
\[
\left| K_n^{(j)}(x)-K^{(j)}(x)\right|\leq \frac{ (j+|x|)}{2n}.
\]
Fix $0< \gamma <1/2$. For $|p|\le n^{\gamma}$ and $(t,s)\in [a,b]^2$, the previous bound entails that, for some $C>0$ 
\[
\left| K_n^{(j)}(p(t-s))-K^{(j)}(p(t-s))\right|\le \frac{ (j+|p (t-s)|)}{2n}\le \frac{C}{n^{1-\gamma}}.
\]
Finally, making the sum for all $|p| \leq n^{\gamma}$, one gets the estimate \eqref{eq.unifK21}
with $\kappa = 1-2\gamma$. Let us now focus on the estimate (\ref{eq.unifK2}). 
As already noted in the proof of Lemma \ref{lem.convuni2}, from Equations \eqref{eq.KnKprimen} and \eqref{eq.Kseconden}, it is here sufficient to lower bound in a polynomial way the term
\[
A_n=A_n(s,t):=n \times \left| \sin \left(\frac{ (p+q)p_n+ (pt+qs)}{2n}\right)\right|.
\]
By the triangle inequality, we have 
\[
A_n \geq n \left|\sin\left( \frac{ (p+q)\alpha}{2}\right)\right|
- n \left|\sin\left( \frac{(p+q)}{2} \left(\frac{p_n}{n}-\alpha\right) + \frac{p t +q s}{2n} \right)\right|.
\]
Due to the concavity of $\sin$ on the interval $[0,\frac{\pi}{2}]$ it follows that we have for every $x\in[0,\frac{\pi}{2}]$ that $|\sin(x)|\ge \frac{2}{\pi} |x|$. Similary, for every $x \in [\frac{\pi}{2},\pi]$ we have $|\sin(x)|\ge \frac{2}{\pi} |\pi-x|$. As a result we obtain, for every $x\in \mathbb{R}$ that $|\sin(x)|\ge \frac{2}{\pi} \text{dist}(x,\pi \mathbb{Z}).$ Taking this into account, and the Diophantine property (\ref{Def-Dioph}) of $\alpha/\pi$, we have then 
\[
\begin{array}{ll}
\displaystyle{\left|\sin\left( \frac{ (p+q)\alpha}{2}\right)\right|} & \displaystyle{\geq \frac{2}{\pi} \text{dist} \left( \frac{ (p+q)\alpha}{2}, \pi \mathbb Z\right)= (p+q)  \text{dist} \left( \frac{ \alpha}{\pi}, \frac{2}{p+q} \mathbb Z\right)}\\
\\
& \geq \displaystyle{(p+q) \times \frac{c_{\alpha}}{(p+q)^{2+\nu_\alpha}} = \frac{c_{\alpha}}{(p+q)^{1+\nu_\alpha}} }.
\end{array}
\]
We then have, for all $0<\gamma<1/2$ and  for every $|p|,|q|\le n^\gamma$ and for a constant $C$ that may change from line to line
\[
n \left|\sin\left( \frac{ (p+q)\alpha}{2}\right)\right| \geq C \times n^{1- \gamma(1+\nu_\alpha)}, \quad \text{where} \,\; 1- \gamma(1+\nu_\alpha)>0.
\]
Besides, since $p_n/n$ goes to $\alpha$ at polynomial rate $c_2$, for all $\gamma<c_2 \wedge 1/2$, we also have uniformly in $(s,t) \in [a,b]^2$
\[
n \left|\sin\left( \frac{(p+q)}{2} \left(\frac{p_n}{n}-\alpha\right) + \frac{p t +q s}{2n} \right)\right| \leq C \left( n^{1+\gamma - c_2} +n^{\gamma}\right).
\]
Therefore, for all exponent $\gamma$ such that
\[
\left \lbrace \begin{array}{l} 0<\gamma < c_2 \wedge 1/2 \\
1- \gamma(1+\nu_{\alpha}) > \gamma \\
1- \gamma(1+\nu_{\alpha})  > 1+ \gamma -c_2
 \end{array} \right. \Longleftrightarrow 0<\gamma < \frac{1}{2+\nu_{\alpha}} \wedge \frac{c_2}{2+\nu_{\alpha}},
\]
as $n$ goes to infinity, we have $A_n \geq C n^{1- \gamma(1+\nu_\alpha)} \times \left( 1+o(1) \right)$ uniformly in $s,t$. In particular, thanks to the estimates  \eqref{eq.KnKprimen} and \eqref{eq.Kseconden}, we deduce that  
\[
\left|K_n^{(j)}\left(  (p+q)p_n+ (pt+qs)\right)\right| = O \left( \frac{1}{n^{1- \gamma(1+\nu_\alpha)} } \right) .
\]
Making the sum for all $|p|,|q| \leq n^{\gamma}$, we get 
 \[
 \sup_{s,t \in [a,b]} \sum_{\substack{p \neq -q \in \mathbb N \\ |p| \leq n^{\gamma}, |q| \leq n^{\gamma}}}
\left|K_n^{(j)}\left(  (p+q)p_n+ (pt+qs)\right)\right| = O \left( \frac{1}{n^{1- \gamma(3+\nu_\alpha)} } \right).
\]
As a conclusion, we get that for all $0<\gamma < \frac{1}{3+\nu_{\alpha}} \wedge \frac{c_2}{2+\nu_{\alpha}}$, we have indeed
\[
 \sup_{s,t \in [a,b]} \sum_{\substack{p \neq -q \in \mathbb N \\ |p| \leq n^{\gamma}, |q| \leq n^{\gamma}}}
\left|K_n^{(j)}\left(  (p+q)p_n+ (pt+qs)\right)\right| = O \left( \frac{1}{n^{\kappa}} \right), \; \text{with} \;\; \kappa = 1- \gamma(3+\nu_\alpha).
\]
\end{proof}

\subsection{Qualitative estimates}
Let us first remark that decomposing $g$ and $h$ (or $g'$ and $h'$) as the sums of their positive and negative parts, it is sufficient to treat the case where $g$ and $h$ (or $g'$ and $h'$) are non negative. 
We first establish the estimate \eqref{eq.conv1}. Since $g$ and $h$ are two continuous $2\pi-$periodic functions which are piecewise $C^1$, their Fourier series  converge normally and we can write 
\[
C_{n}(s,t)  \displaystyle{= \sum_{p,q \in \mathbb Z}\widehat{g}(p) \widehat{h}(q) A_{p,q,n}(s,t)},
\]
with
\[
A_{p,q,n}(s,t)  := \displaystyle{\frac{1}{n} \sum_{k=1}^n  \exp\left( \frac{i k}{n} [ p(p_n+t)+q(p_n+s)]\right)}= K_n((p+q)p_n+pt+qs).
\]
Again, since the Fourier series of $g$ and $h$ converge normally and since $|A_{p,q,n}(s,t)| \leq 1$ uniformly in all its arguments, for any $\varepsilon>0$, there exist an integer $Q=Q_{\varepsilon}$ such that 
\[
\sup_{s,t \in [a,b]}  \left| C_{n}(s,t) - \sum_{\substack{p,q \in \mathbb Z \\ |p|,|q| \leq Q_{\varepsilon}}}\widehat{g}(p) \widehat{h}(q) A_{p,q,n}(s,t)\right| \leq \varepsilon.
\]
By Lemma \ref{lem.convuni2}, uniformly in $s,t \in [a,b]$ and $p\neq -q$ such that $|p|,|q| \leq Q_{\varepsilon}$, we have then 
\[
\displaystyle{ \lim_{n \to +\infty} A_{p,q,n}(s,t)}  = \displaystyle{ \lim_{n \to +\infty} K_n((p+q)p_n+pt+qs)  =0}. 
\]
Moreover, by Lemma \ref{lem.convuni}, uniformly in $s,t \in [a,b]$ and $|p| \leq Q_{\varepsilon}$, we have  also
\[
\displaystyle{ \lim_{n \to +\infty} A_{p,-p,n}(s,t)}  \displaystyle{= \lim_{n \to +\infty} K_n( p(t-s)) = K(p(t-s)).}
\]
Therefore, letting $n$ go to infinity, we get indeed
\[
\limsup_{n \to +\infty} \sup_{s,t \in [a,b]}  \left| C_{n}^{}(s,t) -\sum_{ p \in \mathbb Z}\widehat{g}(p) \widehat{h}(-p)   K(p(t-s))\right| \leq 2\varepsilon.
\]
Let us now concentrate on the estimate \eqref{eq.conv2}. By hypothesis, the functions $g'$ and $h'$ are $2\pi-$periodic and piecewise continuous. From the first remark at the beginning of the proof, we can also suppose without loss of generality that they are non negative. In this case, for all $\varepsilon>0$, there exists $C^{\infty}$ non negative functions  $g_{\varepsilon \pm}'$ and $h_{\varepsilon \pm}'$ such that 
\[
\begin{array}{ll}
g_{\varepsilon^-}'(t) \leq g'(t) \leq g_{\varepsilon^+}'(t) \quad \forall t \in [a,b], & \hbox{and} \quad ||g'-g_{\varepsilon {\pm}}'||_2 \leq \varepsilon, \\
\\
h_{\varepsilon^-}'(t) \leq h'(t) \leq h_{\varepsilon^+}'(t) \quad \forall t \in [a,b], & \hbox{and} \quad ||h'-h_{\varepsilon {\pm}}'||_2 \leq \varepsilon.
\end{array}
\]
and we have then 
\begin{equation}\label{eq.encadre}
 D_n^{\varepsilon-}(s,t)\leq D_n(s,t) \leq D_n^{\varepsilon+}(s,t),
\end{equation}
where 
\[
D_n^{\varepsilon \pm}(s,t)  :=\displaystyle{:= \frac{1}{n} \sum_{k=1}^n \frac{k^2}{n^2}  g_{\varepsilon \pm}' \left( \frac{k(p_n+t)}{n} \right)h_{\varepsilon \pm}' \left( \frac{k(p_n+s)}{n} \right).}
\]
Now, writing the functions $g_{\varepsilon^{\pm}}'$ and $h_{\varepsilon^{\pm}}'$ as the sum of their Fourier series, which converge normally,
\[
g_{\varepsilon^{\pm}}'(t) = \sum_{p \in \mathbb Z} \widehat{g_{\varepsilon^{\pm}}'}(p) e^{i  p t}, \quad h_{\varepsilon^{\pm}}'(t) = \sum_{p \in \mathbb Z} \widehat{h_{\varepsilon^{\pm}}'}(p) e^{i p t},
\]
we have 
\[
D_{n}^{\varepsilon \pm}(s,t)  \displaystyle{= \sum_{p,q \in \mathbb Z}\widehat{g_{\varepsilon^{\pm}}'}(p) \widehat{h_{\varepsilon^{\pm}}'}(q) B_{p,q,n}(s,t)},
\]
with 
\[
\begin{array}{ll}
B_{p,q,n}(s,t) & := \displaystyle{\frac{1}{n} \sum_{k=1}^n \frac{k^2}{n^2} \exp\left( \frac{i  k}{n} [ p(p_n+t)+q(p_n+s)]\right)}\\
\\
& = \displaystyle{- K_n''((p+q)p_n+pt+qs) }.
\end{array}
\]
As above, since the Fourier series converge normally and since $| B_{p,q,n}(s,t)| \leq 1$ uniformly in its arguments, there exists a positive integer $Q_{\varepsilon}$ such that 
\begin{equation}\label{eq.newtronc2}
\sup_{s,t \in [a,b]}  \left| D_{n}^{\varepsilon \pm}(s,t) - \sum_{\substack{p,q \in \mathbb Z \\ |p|,|q| \leq Q_{\varepsilon}}}\widehat{g_{\varepsilon^{\pm}}'}(p) \widehat{h_{\varepsilon^{\pm}}'}(q) B_{p,q,n}(s,t)\right| \leq \varepsilon.
\end{equation}
By Lemmas  \ref{lem.convuni} and \ref{lem.convuni2}, uniformly in $s,t \in [a,b]$ and $p\neq -q$ such that $|p|,|q| \leq Q_{\varepsilon}$, we have then 
\[
\displaystyle{ \lim_{n \to +\infty} B_{p,q,n}(s,t)}  =\displaystyle{  - \lim_{n \to +\infty} K_n''((p+q)p_n+pt+qs)  =0}. 
\]
\[
\displaystyle{ \lim_{n \to +\infty} B_{p,-p,n}(s,t)}  \displaystyle{= -\lim_{n \to +\infty} K_n''( p(t-s)) = - K''(p(t-s)).}
\]
Therefore, letting $n$ go to infinity, we deduce that 
\begin{equation}\label{eq.newtronc4}
\limsup_{n \to +\infty} \sup_{s,t \in [a,b]}  \left| D_{n}^{\varepsilon \pm}(s,t) + \sum_{ p \in \mathbb Z}\widehat{g_{\varepsilon^{\pm}}'}(p) \widehat{h_{\varepsilon^{\pm}}'}(-p)   K''(p(t-s))\right| \leq 2\varepsilon.
\end{equation}
Now, since $K''$ is bounded, we can upper bound the error term as follows
\[
\begin{array}{ll}
\Delta^{\varepsilon} & :=\displaystyle{ \left| \sum_{ p\in \mathbb Z}\widehat{g_{\varepsilon^{\pm}}'}(p) \widehat{h_{\varepsilon^{\pm}}'}(-p)   K''(p(t-s)) - \sum_{ p\in \mathbb Z}\widehat{g'}(p) \widehat{h'}(-p)   K(p(t-s)) \right|} \\
\\
&\displaystyle{ \leq  \sum_{p\in \mathbb Z}| \widehat{g_{\varepsilon^{\pm}}'}(p) \widehat{h_{\varepsilon^{\pm}}'}(-p) -\widehat{g'}(p) \widehat{h'}(-p)| \times | K''(p(t-s))|} \\
\\
& \leq \displaystyle{ \sum_{p\in \mathbb Z}| \widehat{g_{\varepsilon^{\pm}}'}(p) \widehat{h_{\varepsilon^{\pm}}'}(-p) -\widehat{g'}(p) \widehat{h'}(-p)| }\\
\\
& \leq \displaystyle{  \sum_{p\in \mathbb Z} | \widehat{g_{\varepsilon^{\pm}}'}(p) -\widehat{g'}(p) | \times | \widehat{h_{\varepsilon^{\pm}}'}(-p)| +|\widehat{h_{\varepsilon^{\pm}}'}(-p) -\widehat{h'}(-p) | \times |\widehat{g'}(p)|}.
\end{array}
\]
Applying twice the Cauchy-Schwarz inequality and the Parseval identity, one thus gets
\[
\begin{array}{ll}
\Delta^{\varepsilon} 
& \leq \displaystyle{   ||g_{\varepsilon^{\pm}}' -g'||_2 \times ||h_{\varepsilon^{\pm}}'||_2 +   ||h_{\varepsilon^{\pm}}' -h'||_2 \times ||g'||_2  } \\
\\
& \leq \displaystyle{ \varepsilon \left(||g'||_2+ ||h'||_2 +\varepsilon \right)}.
\end{array}
\]
This last estimate associated to Equations \eqref{eq.encadre} and \eqref{eq.newtronc4} yields that, for some explicit constant $C>0$, we have
\begin{equation}\label{eq.newtronc8}
\limsup_{n \to +\infty} \sup_{s,t \in [a,b]}  \left| D_{n}(s,t) + \sum_{ p \in \mathbb Z}\widehat{g'}(p) \widehat{h'}(-p)   K''(p(t-s))\right| \leq C \varepsilon.
\end{equation}
Remembering that $\widehat{g'}(p) = i \times p \times \widehat{g}(p)$, we get
\begin{equation}\label{eq.newtronc9}
\limsup_{n \to +\infty} \sup_{s,t \in [a,b]}  \left| D_{n}(s,t) +\sum_{ p \in \mathbb Z}\widehat{g}(p) \widehat{h}(-p) p^2  K''(p(t-s))\right| \leq C \varepsilon,
\end{equation}

hence the estimate \eqref{eq.conv2}. The proof of \eqref{eq.conv3} is totally similar and left to the reader.

\subsection{Quantitative estimates}
Let us finally focus on the quantitative estimate \eqref{eq.conv4}. Recall that the functions $g'$ and $h'$ are $2\pi-$periodic and piecewise continuous, that can be supposed non negative without loss of generality. For all integer $n>0$, there exists $C^{\infty}$ non negative functions  $g_{n^\pm}'$ and $h_{n^\pm}'$ such that for some $0<\beta<1$ to be determinate later, we have 
\[
\begin{array}{ll}
g_{n^-}'(t) \leq g'(t) \leq g_{n^+}'(t) \quad \forall t \in [a,b], & \hbox{and} \quad ||g'-g_{n {\pm}}'||_2 \leq \frac{1}{n^{\beta}}, \\
\\
h_{n^-}'(t) \leq h'(t) \leq h_{n^+}'(t) \quad \forall t \in [a,b], & \hbox{and} \quad ||h'-h_{n {\pm}}'||_2 \leq \frac{1}{n^{\beta}}.
\end{array}
\]
Note that the $j^{th}$ derivatives of $g_{n^\pm}$ then satisfy the estimates
\begin{equation}
\label{eq.borneFourier}
||  g_{n^\pm}^{(j)} ||_{\infty} = O\left(  n^{\beta j} \right), \; ||  g_{n^\pm}^{(j)} ||_{2} = O\left(  n^{\beta j} \right),
\end{equation}
and similarly for the derivatives of $h_{n^\pm}$.
As above, we have then 
\begin{equation}\label{eq.encadred}
D_n^{-}(s,t)\leq D_n(s,t) \leq D_n^{+}(s,t),
\end{equation}
where 
\[
D_n^{\pm}(s,t)  := \frac{1}{n} \sum_{k=1}^n \frac{k^2}{n^2}  g_{n^\pm} \left( \frac{k(p_n+t)}{n} \right)h_{n^\pm} \left( \frac{k(p_n+s)}{n} \right).
\]
Writing the functions $g_{n^{\pm}}$ and $h_{n^{\pm}}$ as the sum of their Fourier series, and introducing another exponent $0<\gamma<1$ to be fixed later, the term $D_n^{\pm}(s,t)$ can then be decomposed as 
\[
D_{n}^{\pm}(s,t) \displaystyle{= \sum_{p,q \in \mathbb Z}\widehat{g_{n^{\pm}}}(p) \widehat{h_{n^{\pm}}}(q) B_{p,q,n}(s,t)}=P_n(s,t) + Q_n(s,t) +R_n(s,t),
\]
with 
\[
\begin{array}{ll}
P_n(s,t) & := \displaystyle{\sum_{\substack{ p,q \in \mathbb Z \\ |p|, |q|\leq n^{\gamma} \\ p = - q}}\widehat{g_{n^{\pm}}}(p) \widehat{h_{n^{\pm}}}(q) B_{p,q,n}(s,t)}\\
\\
 Q_n(s,t) & := \displaystyle{\sum_{\substack{ p,q \in \mathbb Z \\ |p|, |q|\leq n^{\gamma}\\ p \neq -q}}\widehat{g_{n^{\pm}}}(p) \widehat{h_{n^{\pm}}}(q) B_{p,q,n}(s,t)}\\
\\R_n(s,t) & := \displaystyle{\sum_{\substack{ p,q \in \mathbb Z \\ |p|>n^{\gamma} \, \text{or}\, |q|>n^{\gamma} }}\widehat{g_{n^{\pm}}}(p) \widehat{h_{n^{\pm}}}(q) B_{p,q,n}(s,t)}.
\end{array}
\]
Let us first show that $R_n$ goes to zero at a polynomial rate as $n$ goes to infinity. We can upper bound the sum by

\[
\begin{array}{ll}
|R_n(s,t)| &\displaystyle{\ \leq \sum_{\substack{ p\in \mathbb Z \\ |q|>n^{\gamma}}}|\widehat{g_{n^{\pm}}}(p)\widehat{h_{n^{\pm}}}(q) | +\sum_{\substack{ q\in \mathbb Z \\ |p|>n^{\gamma}}}|\widehat{g_{n^{\pm}}}(p)\widehat{h_{n^{\pm}}}(q) |+\sum_{\substack{ |p|,|q|> n^{\gamma}}}|\widehat{g_{n^{\pm}}}(p)\widehat{h_{n^{\pm}}}(q) |  }\\
\\
& =:R_n^1 + R_n^2 + R_n^3.
\end{array}
\]
We have then 
\[
R_n^1 = \sum_{p\in \mathbb Z} |\widehat{g_{n^{\pm}}}(p)| \sum_{ |q|>n^{\gamma}}|\widehat{h_{n^{\pm}}}(q) |.
\]
On the one hand, writing for $p\neq0$
\[
|\widehat{g_{n^{\pm}}}(p)| = \frac{|\widehat{g_{n^{\pm}}'}(p)| }{|p|} \leq \frac{1}{2} \left( |\widehat{g_{n^{\pm}}'}(p)|^2 +\frac{1}{p^2} \right), 
\]
we get by Parseval identity and the estimates \eqref{eq.borneFourier}
\[
\sum_{p \in \mathbb Z} |\widehat{g_{n^{\pm}}}(p)|=\frac{1}{2} \sum_{p \in \mathbb Z} |\widehat{g_{n^{\pm}}'}(p)|^2 + O(1) = \frac{1}{2} ||g_{n^{\pm}}'||_2^2 + O(1) = O(n^{2\beta})+ O(1).
\]
On the other hand, for $j$ large enough, writing
\[
|\widehat{h_{n^{\pm}}^{}}(p)| = \frac{|\widehat{h_{n^{\pm}}^{(j)}}(p)| }{|p|^j} \leq \frac{O(n^{\beta j}) }{|p|^j},
\]
using again the estimate \eqref{eq.borneFourier}, we get 
\[
\sum_{ |q|>n^{\gamma}}|\widehat{h_{n^{\pm}}}(q) | \leq O(n^{\beta j} n^{-\gamma (j-1)}),
\]
so that $R_n^1 = O( n^{\beta(j+2)-\gamma (j-1)})$ and similarly for $R_n^2$ by symmetry. The last term $R_n^3$ can be treated in the same manner to get 
\[
R_n^3 = O(n^{2\beta j- 2\gamma (j-1)}).
\]
Therefore, as soon as $\beta<\gamma$, for $j$ large enough, we get that $R_n(s,t)$ goes to zero at a polynomial rate in $n$, uniformly in $s,t \in [a,b]$.
We now concentrate on the term $Q_n(s,t)$. We have
\[
|Q_n(s,t)| \leq \frac{||g_{n^{\pm}}||_{\infty} ||h_{n^{\pm}}||_{\infty}}{4\pi^2} \sum_{\substack{ p,q \in \mathbb Z \\ |p|, |q|\leq n^{\gamma}\\ p \neq -q}} | K_n''((p+q)p_n+pt+qs)|.
\]
Chossing $\gamma $ small enough, by the estimate \eqref{eq.unifK2} of Lemma \ref{lem.convuni3}, we get  that there exists $\kappa>0$ such that
\[
\sup_{s,t \in [a,b]} |Q_n(s,t)| = O(n^{-\kappa}).
\]
We are thus left with the term $P_n(s,t)$, which can be rewritten as 
\[
P_{n}(s,t) = -  \sum_{\substack{ |p|\leq n^{\gamma}}}\widehat{g_{n^{\pm}}}(p) \widehat{h_{n^{\pm}}}(-p) K_n''(p(t-s)).
\]
Let us define
\[
\begin{array}{ll}
U_n(s,t) & :=\displaystyle{- \sum_{\substack{ |p|\leq n^{\gamma}}}\widehat{g_{n^{\pm}}}(p) \widehat{h_{n^{\pm}}}(-p) K''(p(t-s))},\\
\\
V_n(s,t)& :=\displaystyle{-  \sum_{\substack{ p \in \mathbb Z}}\widehat{g_{n^{\pm}}}(p) \widehat{h_{n^{\pm}}}(-p) K''(p(t-s))},\\
\\
W(s,t)& :\displaystyle{-  \sum_{\substack{ p \in \mathbb Z}}\widehat{g}(p) \widehat{h}(-p) K''(p(t-s))}.
\end{array}
\]
We have 
\[
|P_{n}(s,t)- U_n(s,t)| \leq {|| g_{n^{\pm}}||_{\infty} || h_{n^{\pm}}||_{\infty}} \sum_{\substack{ |p|\leq n^{\gamma}}}|K_n''(p(t-s))- K''(p(t-s))| 
\]
so that, again choosing $\gamma$ small enough, by the estimate \eqref{eq.unifK21} of Lemma \ref{lem.convuni3}, we get
\[
\begin{array}{ll}
\displaystyle{\sup_{s,t \in [a,b]}|P_{n}(s,t)- U_n(s,t)| } &  \displaystyle{=  O( n^{-\kappa})}.
\end{array}
\]
Next, with the same reasoning as above, we have for all positive integer $j$
\[
\sup_{s,t \in [a,b]} |U_n(s,t) - V_n(s,t)| \leq  \sum_{\substack{ |p|> n^{\gamma}}} |\widehat{g_{n^{\pm}}}(p) \widehat{h_{n^{\pm}}}(-p) | = O\left(n^{2\beta j - \gamma (2j-1)}\right).
\]
Finally, using Cauchy--Schwarz inequality and Parseval identity, we have 
\[
\begin{array}{ll}
\displaystyle{\sup_{s,t \in [a,b]}  | V_{n}(s,t)-W(s,t) | }& \displaystyle{ = O \left(  || g_{n^{\pm}} -g||_2 ||h_{n^{\pm}}||_2 +|| h_{n^{\pm}} -h||_2 ||g||_2 \right)  =O(n^{-\beta})}.
\end{array}
\]
As a conclusion, choosing first $\gamma>0$ small enough so that the conclusion of Lemma \ref{lem.convuni3} holds, and then choosing $0<\beta<\gamma$ accordingly, we have the desired estimate.

\bibliographystyle{alpha}

\par 
\vspace{1cm}

\end{document}